\documentclass[12pt]{amsart}
\usepackage[a4paper,margin=2.9cm]{geometry}
\usepackage{graphicx}
\usepackage{amsmath}
\usepackage{amssymb}
\usepackage{cite}
\usepackage{graphicx}
\usepackage[caption=false]{subfig}
\usepackage{algorithm}
\usepackage{algpseudocode}
\usepackage{multirow}

\title[Conductivity for 2D Incommensurate Bilayers]{Modeling and Computation of Kubo Conductivity for 2D Incommensurate Bilayers}
\author{Simon Etter, Daniel Massatt, Mitchell Luskin, Christoph Ortner}
\date{\today}
\thanks{ML and DM were supported in
  part by ARO MURI Award W911NF-14-1-0247.  CO was supported by ERC Starting Grant
  335120 and Leverhulme Research Project Grant RPG-2017-191. SE and CO acknowledge support for visits to the Institute for Mathematics and
  Its Applications.}

\numberwithin{equation}{section}

\newtheorem{remark}{Remark}

\newtheorem{theorem}{Theorem}
\newtheorem{lemma}{Lemma}
\newtheorem{definition}{Definition}
\newtheorem{corollary}{Corollary}
\newtheorem{assumption}{Assumption}

\newtheorem{example}{Example}
\numberwithin{definition}{section}
\numberwithin{theorem}{section}
\numberwithin{remark}{section}
\numberwithin{prop}{section}
\numberwithin{corollary}{section}
\numberwithin{lemma}{section}
\numberwithin{assumption}{section}
\numberwithin{example}{section}
\numberwithin{conjecture}{section}

\def\mod{{\rm mod}}

\newcommand{\opR}{\mathbf{R}}
\newcommand{\param}{\mathcal{P}}
\newcommand{\parelem}{\zeta}
\newcommand{\Tr}{{\rm Tr}}
\newcommand{\per}{\text{per}}
\newcommand{\C}{\mathcal{C}}

\DeclareMathOperator{\real}{Re}
\DeclareMathOperator{\imag}{Im}

\DeclareMathOperator*{\argmin}{arg\,min}

\newcommand{\acos}{\text{acos}}

\newcommand{\R}{\mathcal{R}}

\newcommand{\A}{\mathcal{A}}

\def\mod{{\rm mod}}

\newcommand{\Real}{\text{Re}}

\newcommand{\ccm}{\bar{\mu}}
\newcommand{\bsigma}{\bar{\sigma}}

\def\eps{\varepsilon}

\begin{document}

\begin{abstract}

This paper presents a unified approach to the modeling and computation of the Kubo conductivity of incommensurate bilayer heterostructures at finite temperature.
Firstly, we derive an expression for the large-body limit of Kubo-Greenwood conductivity in terms of an integral of the conductivity function with respect to a current-current correlation measure. We then observe that the incommensurate structure can be exploited to decompose the current-current correlation measure into local contributions and deduce an approximation scheme which is exponentially convergent in terms of domain size.

Secondly, we analyze the cost of computing local conductivities via Chebyshev approximation. Our main finding is that if the inverse temperature $\beta$ is sufficiently small compared to the inverse relaxation time $\eta$, namely $\beta \lesssim \eta^{-1/2}$, then the dominant computational cost is $\mathcal{O}\bigl(\eta^{-3/2}\bigr)$ inner products for a suitably truncated Chebyshev series, which  significantly improves on the $\mathcal{O}\bigl(\eta^{-2}\bigr)$ inner products required by a naive Chebyshev approximation.

Thirdly, we propose a rational approximation scheme for the low temperature regime $\eta^{-1/2} \lesssim \beta$, where the cost of the polynomial method
increases up to $\mathcal{O}\bigl(\beta^2\bigr),$ but the rational scheme scales
much more mildly with respect to $\beta$.

\end{abstract}

\maketitle


\section{Introduction}
Periodic bilayer 2D heterostructures are typically studied using Bloch Theory
\cite{Kaxiras_2003}. This technique breaks down in the case of {\em
incommensurate} heterostructures, where the ensemble is not periodic, though
each individual sheet may maintain its own periodicity. Previous work introduced
a configuration space representation of incommensurate materials, where
incommensurate systems are classified by local configurations \cite{massatt2017,
carr2017, cances2016}, motivated by concepts introduced in
\cite{Bellissard2002,prodan2012}. The configuration space approach proved to be
useful for numerical simulation of the density of states \cite{carr2017}. In the
present paper, we consider conductivity, which proves to be significantly more
challenging to compute numerically, especially in the low temperature and long
dissipation time regime.
We shall restrict ourselves to the tight-binding model, which has the advantage of being designed for large systems while maintaining accurate quantum information.

Our first main result will be to prove that the Kubo conductivity is well defined in
the thermodynamic limit, as was done for the density of states in
\cite{massatt2017}, and has a similar formulation in terms of configuration
space integrals.  For each local configuration, we compute a local conductivity
using the classical current-current correlation formulation~\cite{Kaxiras_2003}
and then integrate over a compact parametrization of all local configurations.
Specifically, in Theorem~\ref{thm:main}, we obtain an exponential rate of
convergence of the averaged local conductivities to the thermodynamic limit.
Related results have also been obtained within the
framework of C$^*$ algebras~\cite{cances2016} and for a disordered
lattice gas~\cite{prodan2012}, whereas our approach uses the direct matrix
framework developed in \cite{massatt2017}.

Our second main result will be the cost analysis of a linear-scaling conductivity algorithm based on Chebyshev approximation, which is the direct analogue of the Fermi Operator Expansion (FOE) for the density matrix \cite{GC94,GT95} and the Kernel Polynomial Method (KPM) for the density of states \cite{massatt2017,kernel_poly}. Both of these methods expand their respective quantity of interest $q$ in terms of some functional $f(A)$ of the Chebyshev polynomials $T_k(E)$ {applied to} the Hamiltonian matrix $H$,
\[
q = \sum_{k = 0}^\infty c_k \, f\bigl(T_k(H)\bigr)
,
\]
and then truncate this series to a finite set of indices $K = \{0, \ldots, k_\mathrm{max}\}$ for numerical evaluation. This truncation is justified since it can be shown in both cases that the contributions from large matrix powers $k$ decay exponentially.

Unlike the density matrix and the density of states, the conductivity $\sigma$ requires an expansion in terms of pairs of Chebyshev polynomials,
\begin{equation}\label{eqn:conductivity_expansion}
\sigma = \sum_{k_1,k_2 = 0}^\infty c_{k_1,k_2} \, f\bigl(T_{k_1}(H), T_{k_2}(H)\bigr)
,
\end{equation}
and this introduces two new features.
On the one hand, it shifts the main computational burden from evaluating the matrix polynomials $T_k(H)$ to evaluating the functional $f(A,B)$ since the FLOP counts for both operations scale linearly in the size of the Hamiltonian but the two-dimensional nature of the expansion in \eqref{eqn:conductivity_expansion} implies that the number of $f(A,B)$ to evaluate is asymptotically larger than the corresponding number of $T_k(H)$.
On the other hand, \eqref{eqn:conductivity_expansion}
allows for more complex decay behavior of the expansion coefficients $c_{k_1k_2}$ and hence necessitates a more careful analysis of how to choose the truncation indices $K \subset \mathbb{N}^2$.

Indeed, we will see in Section \ref{sec:numerics} that the shape of the large terms in \eqref{eqn:conductivity_expansion} depends heavily on two physical parameters, namely the inverse temperature $\beta$ and the inverse relaxation time $\eta$, and changes from ``wedge along the diagonal'' for $\beta \lesssim \eta^{-1/2}$ to ``equilateral triangle'' for $\beta \gtrsim \eta^{-1}$, see Figure \ref{fig:coeffs},
and the number of significant {terms} changes correspondingly from $\mathcal{O}\bigl(\eta^{-3/2}\bigr)$ for $\beta \lesssim \eta^{-1/2}$ to $\mathcal{O}\bigl(\beta^2\bigr)$ for $\beta \gtrsim \eta^{-1}$, see Table \ref{tbl:Chebyshev_analysis}.
In the case $\beta \gtrsim \eta^{-1}$, we will further see that the number of significant terms can be reduced even further by using a rational approximation instead of \eqref{eqn:conductivity_expansion}.
Since $\beta$ is inversely proportional to the temperature while $\eta$ depends mostly on the material properties \cite{ashcroft2011solid}, the same material at different temperatures can lead to a widely varying relationship between $\beta$ and $\eta$.

An expansion analogous to \eqref{eqn:conductivity_expansion} has previously been considered in \cite{Wan94} for computing optical-absorption spectra.
The main novelty of our work compared to \cite{Wan94} is that we analyze the decay of the terms in \eqref{eqn:conductivity_expansion} and use an adaptive index set $K \subset \mathbb{N}^2$ for truncating this series, while \cite{Wan94} considers only $K = \{0, \ldots, k_\mathrm{max}\}^2$.

\subsection{Notation}

\begin{itemize}
\item We denote the $\ell^2$ norm, the operator norm, and the Frobenius norm over discrete space as $\|\cdot\|_{\ell^2},
 \|\cdot\|_{\rm op}$, $\|\cdot\|_{\rm F}$.
 The supremum norm of a function $f : X \to Y$ on a domain $\Omega \subset X$ is denoted by $\|f\|_\Omega$.
\item $B_r = \{ x \in \mathbb{R}^2 \text{ : } |x| < r\}.$
\item For vectors $v,w \in \mathbb{C}^N$ and $A \in \mathbb{C}^{N\times N}$, we have $\langle v|w \rangle = \sum_{i=1}^N v_i^*w_i$ and $\langle v|A|w\rangle = \sum_{i,j = 1}^N A_{ij}v_i^*w_j.$
\item  $\mathcal{L}(\ell^2(\Omega))$ are the bounded operators from $\ell^2(\Omega)$ to itself.
\item We write ``$f(x) = \mathcal{O}\bigl(g(x)\bigr)$ for $x \to x_0$'' if $\limsup_{x \to x_0} \frac{|f(x)|}{|g(x)|} < \infty$,
and ``$f(x) = \Theta\bigl(g(x)\bigr)$ for $x \to x_0$'' if $\limsup_{x \to x_0} \frac{f(x)}{g(x)} < \infty$ and $\liminf_{x \to x_0} \frac{f(x)}{g(x)} > 0$.
We note that unlike $\mathcal{O}\bigl(g(x)\bigr)$, $\Theta\bigl(g(x)\bigr)$ is signed, i.e.\ $\Theta\bigl(g(x)\bigr) \neq \Theta\bigl(-g(x)\bigr)$.
\end{itemize}


\section{Conductivity in Incommensurate Bilayers}
\label{sec:main}

\subsection{Incommensurate bilayer}
Informally, an {\em incommensurate bilayer} is a union of two infinite sheets of
material, which are individually periodic, but when joined together become
aperiodic (see Fig.~\ref{fig:moire} for an example). To formalize this concept,
let
\begin{equation*}
   \R_\ell := \{ A_\ell m \text{ : } m \in \mathbb{Z}^2\},
\end{equation*}
with non-singular $A_\ell \in \mathbb{R}^{2 \times 2}$, be two Bravais lattices
defining the periodicity of the two sheets indexed by $\ell \in \{1,2\}$.
For future reference, let $\tau(1) = 2, \tau(2) = 1$ denote the transposition operator,
and let
\begin{equation*}
\Gamma_\ell = \{A_\ell \beta \text{ : } \beta \in [0,1)^2\}
\end{equation*}
denote the unit cell for $\R_\ell$. In terms of the reciprocal lattices
\begin{equation*}
   \R_\ell^* := \{ 2\pi A_\ell^{-T} n \text{ : } n \in \mathbb{Z}^2\},
\end{equation*}
we can state the assumption of incommensurability  as follows:

\begin{assumption}
   \label{assump:incom}
   The bilayer $\R_1 \cup \R_2$ is {\em incommensurate}, that is,
   \begin{equation*}
   v + \R_1^*\cup\R_2^* = \R_1^*\cup \R_2^*
   \quad \Leftrightarrow  \quad
   v = (0,0).
   \end{equation*}
\end{assumption}

\begin{figure}
   \includegraphics[width=.7\textwidth]{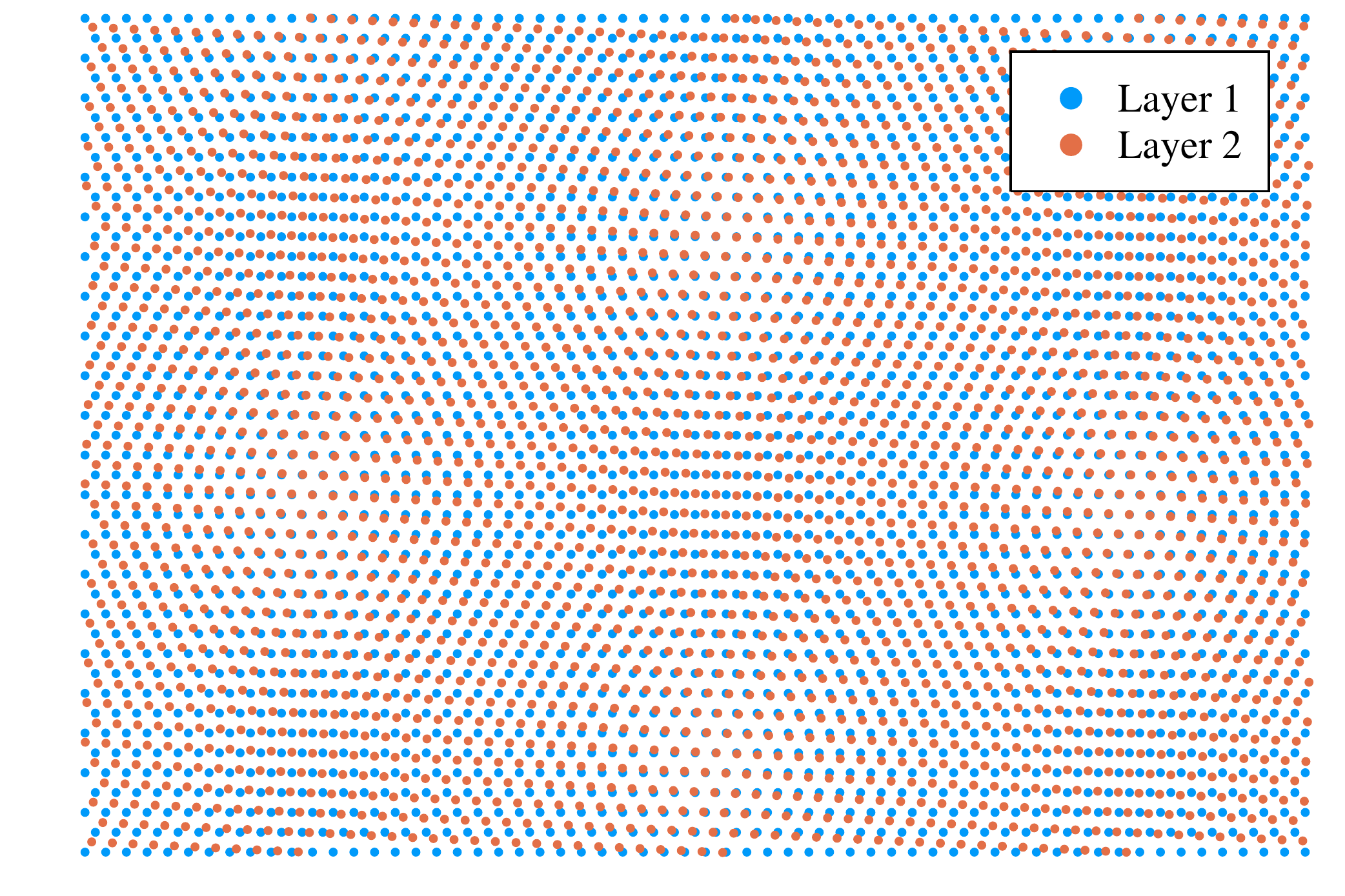}
   \caption{Hexagonal bilayer lattices with a $2.5^\circ$ relative twist.} \label{fig:moire}
\end{figure}

As shown in \cite{massatt2017, cances2016, Zhou2009}, incommensurability leads to a form of
ergodicity that allows us to replace sampling over bilayer sites with sampling
over bilayer shifts or disregistry (henceforth called {\em configurations}; cf. Remark
\ref{rem:configurations}).

\begin{lemma}
\label{thm:ergodic}
   Let $\R_1$ and $\R_2$ satisfy Assumption \ref{assump:incom}, and
   $g \in C_{\rm per}(\Gamma_{\tau(\ell)})$, then
   \begin{equation*}
   \lim_{r \to \infty} \frac{1}{\# \R_\ell \cap B_r} \sum_{R_\ell \in \R_\ell \cap B_r} g(R_\ell) = \frac{1}{|\Gamma_{\tau (\ell)}|} \int_{\Gamma_{\tau (\ell)}} g(b)db,
   \end{equation*}
   where $B_r = \{ x \in \mathbb{R}^2 \text{ : } |x| \leq r\}.$
\end{lemma}

Lemma \ref{thm:ergodic} is the basis of an efficient algorithm for computing
the density of states in incommensurate bilayers \cite{massatt2017}. In the
present work, it plays a similar role in the computation of transport properties.

\begin{remark} \label{rem:configurations}
   The relative shift $b$ between the layers parameterizes the local environment of sites uniquely.  For example, if we let $R \in \R_1$, we have
   \begin{equation*}
   \R_1 \cup \R_2 + R = \R_1 \cup (\R_2 + R) = \R_1 \cup (\R_2 + \mod_2(R)),
   \end{equation*}
   where $\mod_2(R) = R + R' \in \Gamma_2$ for an appropriately chosen $R' \in \R_2$.
   The shift $b = \mod_2(R)$ therefore selects the new environment of
   site $R$, $\R_1 \cup (\R_2 + \mod_2(R))$.

   As a consequence of this observation, we will from now on refer to the shift
   $b$ as a {\em configuration}, and the space of configurations $(\Gamma_1,
   \hspace{1mm}\Gamma_2)$ as {\em configuration space}.
\end{remark}

\subsection{Tight-binding model}
The tight-binding model \cite{Kaxiras_2003} is an electronic structure model,
that has been successfully employed in the modeling of two-dimensional
heterostructures \cite{shiang2015, shiang2016, carr2017 }. For the purpose of the present
work, it will be sufficient to formulate it at an abstract and slightly
simplified level.

Let $\A_\ell$ denote the index set of atomic orbitals for each lattice site of
sheet $\ell$, then the degree of freedom space for the entire bilayer is given
by
\begin{equation}
   \Omega = (\R_1 \times \A_1) \cup (\R_2 \times \A_2).
\end{equation}
(Note that the orbital set $\A_\ell$ also accounts for multi-lattice structures
in the configuration of atomic nuclei.) The tight-binding model is described by
an operator (or, more intuitively, an infinite matrix) $H  \in \mathcal{L}(\ell^2(\Omega))$,
\begin{equation}\label{e:hamiltonian}
H_{R\alpha,R'\alpha'} = h_{\alpha\alpha'}(R-R').
\end{equation}

\begin{assumption}
   \label{assump:decay}
   We assume $h_{\alpha\alpha'} \in C^n(\mathbb{R}^2)$ for some $n > 0$, and is
   exponentially localized for $R= (R_1,R_2) \in \mathbb{R}^2$:
   \begin{equation}  \label{eq:decay_h}
         \begin{split}
   &|h_{\alpha\alpha'}(R)| \lesssim e^{-\gamma_0 |R|}, \\
   &|\partial_{{R_1}}^{m'}\partial_{{R_2}}^{m} h_{\alpha\alpha'}(R)| \lesssim e^{-\gamma_{m'm} |R|},
   \end{split}\end{equation}
   for $\gamma_{m'm} > 0$ and $\gamma_0 > 0$, $m+m' \leq n.$  Further, we assume
   \begin{equation*}
   h_{\alpha\alpha'}(R) = \overline{h_{\alpha'\alpha}(-R)}.
   \end{equation*}
\end{assumption}
Note that $H$ is Hermitian.
{In tight-binding models, the interlayer coupling functions $h$ are smooth \cite{shiang2015, shiang2016} as they are constructed from the coupling between smooth Wannier orbitals.}
Since the infinite-dimensional electronic structure problem (diagonalizing $H$)
cannot be solved directly, we first consider a projection to a finite subset of
the degree of freedom space
\begin{equation}\label{e:Omega_r}
   \Omega_r = \biggl[\bigl[\R_1 \cap B_r\bigl]\times \mathcal{A}_1
           \biggr] \, \cup \, \biggl[\bigl[\R_2 \cap B_r\bigl]\times \mathcal{A}_2\biggr],
            \qquad \text{ for } r > 0.
\end{equation}
Let the projected Hamiltonian be the matrix $H^r = H|_{\Omega_r}$, then we can
solve the corresponding eigenvalue problem
\begin{equation} \label{eq:eval_problem_Hr}
   H^r v_i = \eps_i v_i,
\end{equation}
with $\|v_i\|_{\ell^2} = 1$. A wide range of physical quantities of interest
can be inferred from the eigenpairs $(\eps_i, v_i)$, including electronic
conductivity which we discuss next.

Under Assumption \ref{assump:decay}, the spectrum of $H^r$ is uniformly bounded
as $r \to \infty$. Upon shifting and rescaling the Hamiltonian, we may therefore
assume, without loss of generality, that $\|H\|_{\rm op} < 1$.

\subsection{Current-current correlation measure}
The conductivity tensor will be defined in terms of the {\em current-current
correlation measure}. To introduce it, let $p \in \{1,\,2\}$, and $A \in
\mathbb{R}^{\Omega_r \times \Omega_r}$ be a Hamiltonian. Then the {\em
velocity operator} $\partial_p A \in {\mathbb{C}}^{\Omega_r \times \Omega_r}$
is given by
\begin{equation}
   [\partial_p A]_{R\alpha,R'\alpha'} = i(R'-R)_p A_{R\alpha,R'\alpha'}, \qquad R\alpha,R'\alpha' \in \Omega_r.
\end{equation}
Equivalently, we can define $\partial_p A$ in terms of a commutator, $\partial_p A
= i[A,\opR_p] = i(A \opR_p - \opR_p A )$, where $\opR_p$ is understood as a diagonal matrix
\begin{equation*}
[\opR_p]_{R\alpha,R'\alpha'} = \delta_{\alpha\alpha'}\delta_{RR'} R_p.
\end{equation*}

The matrix-valued current-current correlation measure $\bar{\mu}^r$ on the
finite system $\Omega_r$, is defined by \cite{Combes2010-ec}
\begin{equation} \label{eq:def_currcurrcormeasure_Omegar}
\begin{split}
& \int_{\mathbb{R}^2}\phi(E_1, E_2) d\bar{\mu}^r(E_1,E_2) \\
&\qquad =
\bigg[  \frac{1}{|\Omega_r|} \sum_{i,i'} \phi(\eps_{i}, \eps_{i'}) \, \Tr \Big[ | v_{i} \rangle\langle v_{i} | \partial_{p} H^r | v_{i'} \rangle \, \langle v_{i'} | \partial_{p'} H^r | \Big] \bigg]_{p,p' = 1,2}
\end{split}
\end{equation}
where $(\eps_i, v_i)$ denote the eigenpairs of the Hamiltonian $H^r$, and
$E_1, E_2$ are integration variables. (In particular, the indices  in $E_1, E_2$
are unrelated to the indices of the layers.)

We note that \eqref{eq:def_currcurrcormeasure_Omegar}  is the current-current correlation measure since the current operator $i[\opR_p,A]$
is the negative of the velocity operator $\partial_p A=i[A,\opR_p].$
For the sake
of simplicity of notation, we will henceforth simply drop the brackets
$[\bullet]_{p,p'}$ on the right-hand side of
\eqref{eq:def_currcurrcormeasure_Omegar}.
In numerical computations, we will approximate general functions $\phi(E_1,E_2)$  by sums of products of univariate functions
\begin{equation*}
\phi(E_1,E_2) \approx \tilde{\phi}(E_1,E_2) := \sum_{(k_1, k_2) \in K} \phi_{k_1}(E_1)\phi_{k_2}(E_2),
\end{equation*}
where $K$ is a finite index-set.
In this case, we can rewrite \eqref{eq:def_currcurrcormeasure_Omegar}
(with $\phi$ replaced with $\tilde{\phi}$) as
\begin{equation} \label{eq:currcurr_tensor}
{\int_{\mathbb{R}^2} \tilde\phi(E_1,E_2) \, d\bar{\mu}^r(E_1,E_2) =  \biggl[\frac{1}{|\Omega_r|} \sum_{(k_1, k_2) \in K} \, \Tr \bigl[\phi_{k_1}(H_r)  \partial_{p} H^r  \phi_{k_2}(H_r)  \partial_{p'} H^r \bigr]\biggr]_{p,p' = 1,2}.}
\end{equation}

For brevity we collect the set of conductivity parameters $\parelem = (\beta, \eta,
\omega, E_F) \in \param = \mathbb{R}_+^2 \times \mathbb{R}^2.$ The conductivity
tensor for the finite system $\Omega_r$ can now be defined by
{
\begin{equation} \label{eq:defn_sigma_r}
   \bsigma^r =  \int_{\mathbb{R}^2} F_{\parelem}(E_1,E_2)d\ccm^r(E_1,E_2),
\end{equation}
for the {\em conductivity function} $F_{\parelem}$ defined as
\begin{equation}\label{e:conductivity_function}
   F_{\parelem}(E_1,E_2) = \frac{f_\beta(E_1-E_F) - f_\beta(E_2-E_F)}{(E_1-E_2)(E_1-E_2 + \omega + \iota \eta)}.
\end{equation}
where $\omega$ is proportional to photon frequency, $\eta$ is proportional to inverse relaxation time, $E_F$ is the Fermi-level of the
system, and $f_\beta(E) = (1+ e^{\beta (E-E_F)})^{-1}$ is the Fermi-Dirac
distribution.
Here we have rescaled $\eta$, $\beta$, and all energies to be unitless, and the conductivity is missing a physical constant prefactor.
We note that for a finite system, this is not a true conductivity.  Conductivity is defined only in the infinite system, and hence for the finite system this is an approximate conductivity, which we analyze in this text.

Our aim throughout the remainder of Section \ref{sec:main} is to show that the
thermodynamic limit $\sigma := \lim_{r\rightarrow\infty} \bsigma^r$ exists and
to establish a configuration space representation with an exponential
convergence rate.
}
\begin{remark}
   The formulation \eqref{eq:defn_sigma_r} is consistent with the formulation
   for periodic systems \cite{Kaxiras_2003} and with the C$^*$ algebra
   formulation of a generalized Kubo formula for incommensurate bilayers
   \cite{cances2016}. We will obtain a definition through a thermodynamic limit
   argument using a direct matrix formulation, thus giving this formulation
   additional justification. Here we focus on the thermodynamic limit taken as a
   sequence of circular domains, though we observe that this could be  extended
   to a more general class of limit sequences. In particular, as long as the
   sequence does not generate a proportionally imbalanced boundary relative to
   bulk, the sequence will converge to the same limit. We restrict ourselves to the
   circular domain limit to avoid distraction from the key points of this paper.
\end{remark}

Implicitly, $\bsigma^r$ and later $\sigma$ depend on the model parameters $\parelem =
(\beta, \eta, \omega, E_F) $, but for the sake of brevity of notation, this
dependence is suppressed. However, we emphasize that for a quantitative
convergence analysis the parameters $\beta, \eta$ are in fact crucial since they
characterize the region of analyticity of the conductivity function $F_{\parelem}$.

\subsection{Local current-current correlation measure}
In order to pass to the limit as $r \to \infty$, we follow the ideas in
\cite{massatt2017} and define a local (or, projected) conductivity, which
will later take the role of $g$ in Lemma \ref{thm:ergodic}. To motivate, we first observe that the expression in \eqref{eq:def_currcurrcormeasure_Omegar} can be written as
\begin{equation*}
\begin{split}
\int_{\mathbb{R}^2}\phi(E_1,E_2) \, d\bar{\mu}^r(E_1,E_2) &=
  \frac{1}{|\Omega_r|} \sum_{i,i' } \phi(\eps_{i}, \eps_{i'}) \Tr \bigl[| v_{i} \rangle\langle v_{i} | \partial_{p} H^r | v_{i'} \rangle \, \langle v_{i'} | \partial_{p'} H^r |\bigr] \\
&\hspace{-2cm}=  \frac{1}{|\Omega_r|}\sum_{R\alpha \in \Omega_r}\biggl[\sum_{i,i' }\phi(\eps_{i}, \eps_{i'}) \langle e_{R\alpha}| v_{i} \rangle\langle v_{i} | \partial_{p} H^r | v_{i'} \rangle \, \langle v_{i'} | \partial_{p'} H^r |e_{R\alpha} \rangle\biggr].
\end{split}
\end{equation*}
Here we have defined $e_{R\alpha} \in \ell^2(\Omega_r)$ via
\begin{equation*}
[e_{R\alpha}]_{R'\alpha'} = \delta_{\alpha\alpha'}\delta_{RR'}, \hspace{3mm} R'\alpha' \in \Omega_r,
\end{equation*}
and $(\eps_i,v_i)$ are the eigenpairs of $H^r$. We see that the trace is decomposed
into projections onto diagonal elements. We further observe that the left-most
sum, normalized by $\frac{1}{|\Omega_r|}$, looks remarkably similar to a
discretized integral. The crucial step then is how to realize the thermodynamic
limit as an integral. We will formalize this with the help of Lemma
\ref{thm:ergodic}, which will convert this expression into an integral over
configuration space. To that end, we define the Hamiltonian for a shifted
configuration,
\begin{equation}
[H_{\ell}(b)]_{R\alpha,R'\alpha'} = h_{\alpha\alpha'}\bigl(b(\delta_{\alpha \in \mathcal{A}_{\tau(\ell)}} - \delta_{\alpha' \in \mathcal{A}_{\tau(\ell)}})+R-R'\bigr), \quad R\alpha,R'\alpha' \in \Omega.
\end{equation}
Likewise, we have $H_\ell^r(b) = H_\ell(b)|_{\Omega_r}.$ Since $H^r_{\ell}(b)$
is {Hermitian}, we can define the local current-current correlation measure
$\mu_\ell^r[b]$ for a finite system $\Omega_r$, at configuration $b$, in layer
$\ell$, via
\begin{equation}
\label{e:finite}
\int_{\mathbb{R}^2}\phi(E_1,E_2) \, d\mu^r_\ell[b] =  \sum_{\substack{i,i' \\ \alpha \in \A_\ell}} \phi(\eps_{i}, \eps_{i'}) \, \langle e_{0 \alpha} | v_i \rangle \langle v_i | \partial_p H^r_{\ell}(b) | v_{i'} \rangle \, \langle v_{i'} | \partial_{p'} H^r_{\ell}(b) | e_{0 \alpha} \rangle,
\end{equation}
where $(\eps_i,v_i)$ are the eigenpairs of $H^r_{\ell}(b)$ (and thus implicitly depend on $r, \ell,$ and $b$).

Our next result states that  $\lim_{r \to \infty} \mu^r_\ell[b]$
is well-defined. To that end, we first define a strip in the complex plane
\begin{equation*}
    S_a = \{z \mid \real(z) \in [-a-1, a+1], \imag(z) \in [-a,a]\}.
\end{equation*}

\begin{lemma} \label{thm:local}
Under Assumptions \ \ref{assump:incom} and \ref{assump:decay}, there exist unique
measures $\mu_\ell[b], \ell = 1,2,$ such that for all $F$ that are analytic on $S_{a} \times S_{a}$,
\begin{equation*}
\int_{\mathbb{R}^2}F(E_1,E_2)d\mu_{\ell}^r[b](E_1,E_2) \rightarrow  \int_{\mathbb{R}^2} F(E_1,E_2) d\mu_{\ell}[b](E_1,E_2)
\end{equation*}
with the rate
\begin{equation*}\begin{split}
   \left|\int_{\mathbb{R}^2}F(E_1,E_2)d\mu_{\ell}^r[b](E_1,E_2) -  \int_{\mathbb{R}^2} F(E_1,E_2) d\mu_{\ell}[b](E_1,E_2)\right| &\\
    \lesssim \sup_{z,z' \in S_{a} \setminus S_{a/2}}|F(z,z')| e^{-\gamma a r - c \log(a)},&
\end{split} \end{equation*}
for some $c, \gamma > 0$. {Furthermore, we have the maps
\begin{align*}
   &b \mapsto \int_{\mathbb{R}^2} F(E_1,E_2) d\mu_{\ell}^r[b](E_1,E_2) \in C^n(\Gamma_{\tau(\ell)}), \qquad \text{and} \\
   &b \mapsto \int_{\mathbb{R}^2} F(E_1,E_2) d\mu_{\ell}[b](E_1,E_2)
   \in C_\per^n(\Gamma_{\tau(\ell)}).
\end{align*}}
\end{lemma}

Combining Lemma \ref{thm:local} and Lemma \ref{thm:ergodic}, we are now
ready to define  the thermodynamic limit of the current-current correlation
measure and associated conductivity tensor by
\begin{align}
   \notag
   \mu &= \nu\biggr( \int_{\Gamma_2} \mu_1[b] \,db
                   + \int_{\Gamma_1} \mu_2[b] \,db\biggl) \qquad \text{and} \\
  \label{eq:conductivity}
   \sigma &= \int F_{\parelem} \, d\mu(E_1,E_2),
\end{align}
where
\begin{equation*}
\nu = \frac{1}{|\Gamma_1|\cdot |\A_1| + |\Gamma_2|\cdot |\A_2|}.
\end{equation*}
Moreover, we propose an alternative approximation to $\mu$ that exploits
the configuration integrals, and the corresponding approximation
of the conductivity,
\begin{align}
   \notag
   \mu^r &= \nu \biggr( \int_{\Gamma_2} \mu_1^r[b] \, db
               + \int_{\Gamma_1} \mu_2^r[b] \, db \biggl), \qquad \text{and} \\
   \label{e:finite_conductivity}
   \sigma^r &= \int F_{\parelem}  \, d\mu^r(E_1,E_2).
\end{align}
With these definitions, we can state our first main result.

\begin{theorem}
   \label{thm:main}
   Let Assumptions \ref{assump:incom} and \ref{assump:decay} be satisfied,
   then
   \begin{equation*}
      \bar{\sigma}^r \to \sigma \qquad \text{and} \qquad
      \sigma^r \to \sigma \qquad \text{as } r \to \infty.
   \end{equation*}
   More precisely,
   if $\lambda = \min\{ \eta, \beta^{-1}\}$, then there exist
   constants $c, \gamma > 0,$ independent of $\lambda$ and $r$,
   such that
   \begin{equation*}
      |\sigma - \sigma^r| \lesssim e^{-\gamma \lambda r - c \log(\lambda)}.
   \end{equation*}
\end{theorem}

\begin{remark}
   Although we prove convergence of $\bar{\sigma}^r \to \sigma$,
   we do not obtain a rate. Indeed, as a supercell-like approximation
   of an incommensurate system this sequence is expected to converge
   slowly~\cite{cauchyborn17}. Here, $\bar{\sigma}^r$ has error proportional to $(\eta r)^{-1}$ from the boundary effects, as the error of the domain edge site contributions do not decay. This is poor decay compared to the exponential convergence found in the\ $\sigma^r$ scheme \eqref{e:finite_conductivity}. For the development of a numerical algorithm
   (see Section \ref{sec:numerics}), we therefore use the expression for
   $\sigma^r$ as a starting point, where large domain sizes $r$
   are replaced by an (embarrassingly parallel) integration over local
   configurations. { We note that the convergence rate for the effect of a local perturbation in a crystal can often be improved by
more sophisticated boundary conditions \cite{PhysRevB.45.6543}.  However, the perturbation due to incommensurability in 2D bilayers is
global, but we have shown that an exponential rate of convergence can nonetheless be achieved by integration over local configuration.}
\end{remark}


\section{Linear Scaling Algorithm for Local Conductivities}\label{sec:numerics}

We have seen in Section \ref{sec:main} that the conductivity of an infinite incommensurate bilayer can be written as
\begin{equation}\label{e:conductivity_integral}
\sigma
=
\lim_{r \to \infty}
\sigma^r
=
\lim_{r \to \infty}
\nu \left(\int_{\Gamma_2} \sigma^r_1[b] \, db + \int_{\Gamma_1} \sigma^r_2[b] \, db\right)
\end{equation}
where the local conductivities $\sigma^r_\ell[b]$ are given by
\begin{align}
    \sigma^r_\ell[b]
    &:=
    \int F_{\zeta}(E_1,E_2) \, d\mu^r_\ell[b](E_1,E_2)
    \nonumber
    \\&=
    \sum_{i_1,i_2}
    F_\zeta(\varepsilon_{i_1},\varepsilon_{i_2})
    \,
    \langle v_{i_1} | \partial_p H^r_\ell(b) | v_{i_2}\rangle
    \,
    \langle v_{i_2} | \partial_{p'} H^r_\ell(b) | e_{0\alpha} \rangle \langle e_{0\alpha} | v_{i_1} \rangle
    \label{e:local_conductivity}
    .
\end{align}
This section will present a method for evaluating the local conductivities $\sigma^r_\ell[b]$ based on polynomial and rational approximation of the conductivity function $F_\zeta(E_1,E_2)$.
When combined with any off-the-shelf quadrature rule for evaluating the integrals over $\Gamma_1$, $\Gamma_2$ in \eqref{e:conductivity_integral} (e.g.\ the periodic trapezoidal rule, see Subsection~\ref{ssec:quad_convergence}), our method gives rise to a conductivity algorithm which involves three limits: 1) the number of terms in the approximation of $F_\zeta(E_1,E_2)$ going to infinity, 2) the number of quadrature points in \eqref{e:conductivity_integral} going to infinity, and 3) the localization radius $r$ going to infinity.

The main feature of the local conductivity algorithm proposed in this section is that it scales linearly in the number of explicitly represented degrees of freedom $|\Omega_r|$.
It is this linear scaling which sets our algorithm apart from the more straightforward approach of diagonalizing $H$ and inserting the resulting eigenvalues $\varepsilon_i$ and -vectors $v_i$ into \eqref{e:local_conductivity}, which would scale cubically in $|\Omega_r|$, but we caution that the ``linear-scaling'' label is also somewhat misleading since $\Omega_r$ (or equivalenty, $r$) is not an independent variable but rather should be chosen as a function of $\beta$ and $\eta$; cf.\ Theorem \ref{thm:main}.
We will further elaborate on this point in Remark \ref{rem:flop_counts} where we compare our algorithm and the diagonalization algorithm based on their overall scaling with respect to $\beta$ and $\eta$.

As mentioned, the focus of this section is to compute a single local conductivity $\sigma_\ell^r[b]$ for fixed values of the localization radius $r$, sheet index $\ell$ and bilayer shift $b$.
We therefore reduce the notational clutter by introducing the abbreviations
\[
H_\mathrm{loc} := H^{r}_\ell(b)
,\qquad
M_p^\mathrm{loc} := \partial_p H^{r}_\ell(b)
.
\]

\subsection{Algorithm outline}
\label{ssec:algorithm}

Let us consider an approximate conductivity function $\tilde F_{\zeta}$ obtained by truncating the Chebyshev series of $F_{\zeta}$,
\begin{align}
    \tilde F_\zeta(E_1,E_2)
    &:=
    \sum_{(k_1,k_2) \in K}
    c_{k_1 k_2} \, T_{k_1}(E_1) \, T_{k_2}(E_2)
    \label{eqn:truncated_chebyshev_series}
    \\&\,\approx
    \,\,\,\sum_{k_1,k_2 = 0}^\infty\,\,
    c_{k_1 k_2} \, T_{k_1}(E_1) \, T_{k_2}(E_2)
    =
    F_\zeta(E_1,E_2)
    \label{eqn:chebyshev_series}
\end{align}
where $K \subset \mathbb{N}^2$ is a finite set of indices and $T_k(E)$ denotes the $k$th Chebyshev polynomial defined through the three-term recurrence relation
\begin{equation}\label{eqn:chebrecursion}
T_0(x) = 1, \quad
T_1(x) = x, \quad
T_{k+1}(x) = 2x \, T_k(x) - T_{k-1}(x)
.
\end{equation}
Inserting \eqref{eqn:truncated_chebyshev_series} into \eqref{e:local_conductivity},
we obtain an approximate local conductivity
\begin{align}
    \tilde \sigma_\ell^r[b]
    &:=
    \sum_{i_1,i_2}
    \tilde F_\zeta(\varepsilon_{i_1},\varepsilon_{i_2})
    \,
    \langle v_{i_1} | M^\mathrm{loc}_p | v_{i_2}\rangle
    \,
    \langle v_{i_2} | M^\mathrm{loc}_{p'} | e_{0\alpha} \rangle \langle e_{0\alpha} | v_{i_1} \rangle
    \nonumber
    \\&=
    \sum_{i_1,i_2}
    \sum_{(k_1,k_2) \in K}
    c_{k_1 k_2}
    \,
    \langle e_{0\alpha} | v_{i_1} \rangle
    \,
    T_{k_1}(\varepsilon_{i_1})
    \,
    \langle v_{i_1} | M^\mathrm{loc}_p | v_{i_2}\rangle
    \,
    T_{k_2}(\varepsilon_{i_2})
    \,
    \langle v_{i_2} | M^\mathrm{loc}_{p'} | e_{0\alpha} \rangle
    \nonumber
    \\&=
    \sum_{(k_1,k_2) \in K}
    c_{k_1k_2} \,
    \Big(T_{k_1}(H_\mathrm{loc}) \, M^\mathrm{loc}_p \, T_{k_2}(H_\mathrm{loc}) \, M^\mathrm{loc}_{p'}\Big)_{0\alpha,0\alpha}
    \label{eqn:approximate_local_conductivity}
\end{align}
which can be evaluated without computing the eigendecomposition as follows.
\begin{algorithm}[H] 
    \caption{Local conductivity via Chebyshev approximation}
    \label{alg:local_conductivity}
    \begin{algorithmic}[1]
        \State $\displaystyle | v_{k_1} \rangle := M^\mathrm{loc}_p \, T_{k_1}(H_\mathrm{loc}) \, | e_{0\alpha} \rangle$ for all $k_1 \in K_1 := \{k_1 \mid \exists k_2 : (k_1,k_2) \in K\}$.
        \label{alg:local_conductivity:v}

        \vspace{0.5em}

        \State $\displaystyle | w_{k_2} \rangle := T_{k_2}(H_\mathrm{loc}) \, M^\mathrm{loc}_{p'} \, | e_{0\alpha} \rangle$ for all $k_2 \in K_2 := \{k_2 \mid \exists k_1 : (k_1,k_2) \in K\}$.
        \label{alg:local_conductivity:w}

        \vspace{0.5em}

        \State $\displaystyle
            \tilde \sigma_\ell^r[b]
            :=
            \sum_{(k_1,k_2) \in K} c_{k_1k_2} \, \langle v_{k_1} | w_{k_2} \rangle.
            $
        \label{alg:local_conductivity:vw}
    \end{algorithmic}
\end{algorithm}

Lines \ref{alg:local_conductivity:v} and \ref{alg:local_conductivity:w} of Algorithm \ref{alg:local_conductivity} take $|K_{1}|$ and $|K_2|$, respectively, matrix-vector products when evaluated using the recurrence relation \eqref{eqn:chebrecursion},
while Line \ref{alg:local_conductivity:vw} requires $|K|$ inner products.
Due to the sparsity of $H_\mathrm{loc}$, both types of products take $\mathcal{O}\bigl(|\Omega_r|\bigr)$ floating-point operations, thus we conclude that Algorithm \ref{alg:local_conductivity} scales linearly in the matrix size $|\Omega_r|$.
Furthermore, the error in the computed local conductivity $\tilde\sigma_\ell^r[b]$ can be estimated in terms of the dropped Chebyshev coefficients $c_{k_1k_2}$ as follows.

\begin{lemma}\label{lem:conductivity_approx_bound}
    It holds
    \[
    \big| \tilde \sigma_\ell^r[b] - \sigma_\ell^r[b] \big|
    \lesssim
    \sum_{(k_1,k_2) \in \mathbb{N}^2 \setminus K}
    |c_{k_1k_2}|
    .
    \]
\end{lemma}
\begin{proof}
    The bound follows immediately from \eqref{eqn:approximate_local_conductivity} after noting that
    $M^\mathrm{loc}_p$ and $T_{k}(H_\mathrm{loc})$ are bounded for $p \in \{1,2\}$ and all $k \in \mathbb{N}$.
\end{proof}

A more careful analysis of Algorithm \ref{alg:local_conductivity} reveals that since $|K_{1}|, |K_{2}| \leq |K|$ and both matrix-vector and inner products take $\mathcal{O}(|\Omega_r|)$ floating-point operations, the computational cost of this algorithm is dominated by the cost of Line \ref{alg:local_conductivity:vw} which is $|K|$ inner products.
In the light of Lemma \ref{lem:conductivity_approx_bound}, a good choice for the set $K$ is
\[
K(\tau) := \big\{ (k_1,k_2) \in \mathbb{N}^2 \mid |c_{k_1k_2}| \geq \tau \big\}
\]
for some truncation tolerance $\tau$; thus $|K|$ is linked to the decay of the Chebyshev coefficients which in turn depends on the analyticity properties of $F_\zeta$.
To analyze these, it is convenient to split the conductivity function $F_\zeta(E_1,E_2) = f_\mathrm{temp}(E_1,E_2) \, f_\mathrm{relax}(E_1,E_2)$ into the two factors
\begin{equation}\label{eqn:fermi_factor}
    f_\mathrm{temp}(E_1,E_2)
    :=
    \frac{f_\beta(E_1 - E_F) - f_\beta(E_2 - E_F)}{E_1 - E_2}
\end{equation}
and
\begin{equation}\label{eqn:relaxation_factor}
f_\mathrm{relax}(E_1,E_2)
:=
\frac{1}{E_1 - E_2 + \omega + \iota \eta}
,
\end{equation}
which are easily seen to be analytic\footnote{A precise definition of analyticity in two dimensions will be provided in Definition~\ref{def:analyticity_2d}.} everywhere except, respectively, on the sets
\begin{equation}\label{eqn:fermi_singularities}
S_\mathrm{temp}
:=
\big( S^{(1)}_\mathrm{temp} \times \mathbb{C} \big)
\cup
\Big( \mathbb{C} \times S^{(1)}_\mathrm{temp} \Big)
\quad
\text{with}
\quad
S^{(1)}_\mathrm{temp} := \big\{E_F + \tfrac{\iota \pi k}{\beta} \mid k \text{ odd}\big\}
\end{equation}
and
\begin{equation}\label{eqn:relaxation_singularities}
S_\mathrm{relax} := \big\{(E_{1},E_2) \in \mathbb{C}^2 \mid E_1 - E_2 + \omega + \iota \eta = 0 \big\}.
\end{equation}
The conductivity function $F_\zeta$ is thus analytic except on the union of these two sets.

In one dimension, it is well known that the Chebyshev coefficients $c_k$ of a function $f(x)$ analytic on a neighborhood of $[-1,1]$ decay exponentially, $|c_k| \leq C \, \exp(-\alpha \, k)$, and the decay rate $\alpha$ is equal\footnote{More precisely, it is the \emph{asymptotic} rate of decay which is equal to the parameter of the ellipse of analyticity. Further details are provided in Appendix \ref{sec:proofs_numerics}.} to the parameter $\alpha$ of the largest Bernstein ellipse
\begin{equation}\label{eqn:Bernstein_ellipses}
    E(\alpha)
    :=
    \Big\{
    \cosh(\tilde\alpha) \, \cos(\theta)
    +
    \iota \,
    \sinh(\tilde\alpha) \, \sin(\theta)
    \big)
    \mid
    \tilde \alpha \in [0,\alpha), \theta \in [0,2\pi)
    \Big\}
\end{equation}
which can be inscribed into the domain of analyticity of $f$.
In two dimensions, we have two decay rates $\alpha_1,\alpha_2$ and in the case of the conductivity function $F_\zeta$ we have two sets of singularities $S_\mathrm{temp}$, $S_\mathrm{relax}$ limiting the possible values of $\alpha_1$ and $\alpha_2$.
This suggests to partition the space of parameters $\zeta$ into \emph{relaxation-constrained}, \emph{mixed-constrained}, and \emph{temperature-constrained} depending on whether two, one, or zero of the decay rates are constrained by the singularities $S_\mathrm{relax}$ rather than $S_\mathrm{temp}$.
In Subsection~\ref{ssec:decay_analysis}, we will characterize these parameter regimes more precisely and present asymptotic estimates regarding the number of significant Chebyshev coefficients in each case. A summary of our findings is provided in Table \ref{tbl:Chebyshev_analysis}.
We see that for fixed $\eta$, the cost of Algorithm \ref{alg:local_conductivity} gradually increases from $\mathcal{O}\bigl(\eta^{-3/2}\bigr)$ to $\mathcal{O}\bigl(\beta^2\bigr)$ for increasing inverse temperature $\beta$ which renders conductivity calculations at low temperatures (i.e.,\ large $\beta$) particularly expensive.
In Subsection \ref{ssec:pole_expansion}, we present an alternative algorithm based on a pole expansion of $F_\zeta$ which provably reduces the cost of evaluating the local conductivity to $\mathcal{O}\bigl(\beta^{1/2} \, \eta^{-5/4}\bigr)$ inner products for all $\beta \gtrsim \eta^{-1/2}$ and whose actual scaling was empirically found to be $\mathcal{O}\bigl(\beta^{1/2} \, \eta^{-1.05}\bigr)$ inner products (see~\eqref{eqn:pole_expansion_asymptotics}).

\def\TT{\rule{0pt}{2.6ex}}       
\def\BB{\rule[-1.2ex]{0pt}{0pt}} 

\begin{table}
    \caption{Classification of conductivity parameters $\zeta$ and number of significant terms {(up to logarithmic factors of $\beta$ and $\eta$)} in the Chebyshev series of $F_\zeta$.}
    \label{tbl:Chebyshev_analysis}
    \begin{tabular}{| c | c | c |}
        \hline
        Constraint & Parameter range & $\#$ significant terms \TT\BB \\
        \hline \hline
        Relaxation & $\beta \lesssim \eta^{-1/2}$ & $\mathcal{O}\bigl(\eta^{-3/2}\bigr)$  \TT\BB\\
        \hline
        Mixed & $\eta^{-1/2} \lesssim \beta \lesssim \eta^{-1}$ & $\mathcal{O}\bigl(\beta \eta^{-1}\bigr)$ \TT\BB\\
        \hline
        Temperature & $\eta^{-1} \lesssim \beta$ & $\mathcal{O}\bigl(\beta^2\bigr)$ \TT\BB\\
        \hline
    \end{tabular}
\end{table}

\subsection{Chebyshev coefficients of the conductivity function}\label{ssec:decay_analysis}

A convenient way to visualize the set $S_\mathrm{relax}$ from \eqref{eqn:relaxation_singularities} is to draw two copies of the interval $[-1,1]$ with a shift $\omega + \eta\iota$ between them (the green and blue lines in Figure \ref{fig:singularities}), and
the singularities $S_\mathrm{temp}$ from \eqref{eqn:fermi_singularities} can be added to this picture by drawing a copy of $S^{(1)}_\mathrm{temp}$ relative to each of these intervals (the green and blue dots in Figure \ref{fig:singularities}).
We will see in Appendix \ref{sec:proofs_numerics} that the decay of the Chebyshev coefficients of $F_\zeta(E_1,E_2)$ is determined by the size of the ellipses $E(\alpha_1)$, $E(\alpha_2)$ which can be drawn around the two copies of $[-1,1]$ subject to the following constraints.
\begin{enumerate}
    \item\label{rule:relax}Neither ellipse may contain the endpoints of the other copy of $[-1,1]$.
    \item\label{rule:temp}Neither ellipse may contain any of the points in its copy of $S^{(1)}_\mathrm{temp}$.
    \item\label{rule:overlap}The two ellipses may not overlap if their parameters $\alpha_1, \alpha_2$ are both positive. However, we will see that it is possible for one of the parameters to assume a negative effective value, in which case overlap is admissible (see Figure \ref{fig:contours_relaxation_constrained}).
\end{enumerate}

\begin{figure}
    \centering
    \includegraphics{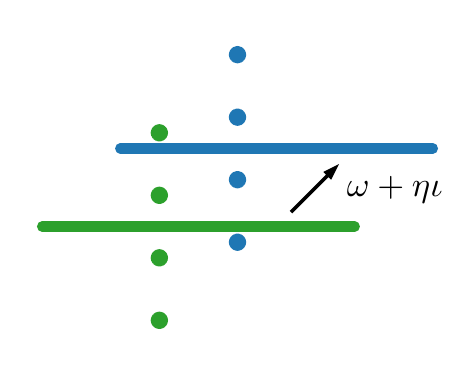}
    \caption{
    Singularities $S_\mathrm{relax} \cup S_\mathrm{temp}$ of the conductivity function $F_\zeta(E_1,E_2)$.
    The solid lines indicate two copies of $[-1,1]$ shifted by $\omega + \iota\eta$ relative to each other, and the dots indicate the set $S^{(1)}_\mathrm{temp}$ relative to the interval of the same color.
    }
    \label{fig:singularities}
\end{figure}

\begin{figure}
    \centering
    \subfloat[Relaxation]{\label{fig:contours_relaxation_constrained}{\includegraphics{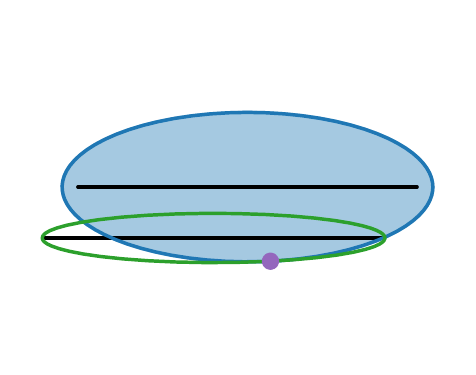}}}
    \hspace{0.2em}
    \subfloat[Mixed]{\label{fig:contours_mixed_constrained}{\includegraphics{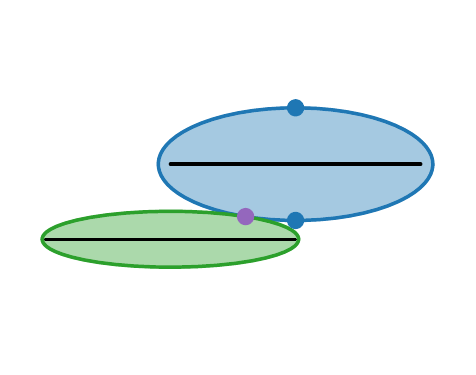}}}
    \hspace{0.2em}
    \subfloat[Temperature]{\label{fig:contours_temperature_constrained}{\includegraphics{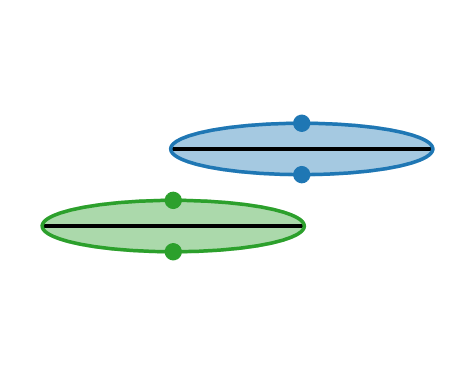}}}
    \caption{
    Ellipse pairs for relaxation-, mixed- and temperature-constrained parameters. The blue and green dots indicate the points in $S_\mathrm{temp}$ restricting the ellipses.
    The purple dots indicate the $x^\star(\zeta)$ introduced in \eqref{eqn:xstar}.
    }
    \label{fig:contours_with_temperature}
\end{figure}

Let us now determine pairs of ellipses by first choosing the upper (blue) ellipse as large as possible subject to Rules \ref{rule:relax} and \ref{rule:temp}, and then maximizing the lower (green) ellipse subject to Rules \ref{rule:temp} and \ref{rule:overlap} for the given upper ellipse.
This procedure allows us to distinguish the relaxation-, mixed- and temperature-constrained parameters $\zeta$ as follows.
\begin{itemize}
    \item Relaxation-constrained: $\beta$ is small enough such that Rule \ref{rule:relax} restricts the upper ellipse. See Figure \ref{fig:contours_relaxation_constrained}.
    \item Mixed-constrained: $\beta$ is large enough such that Rule \ref{rule:temp} restricts the upper ellipse, but it is small enough such that Rule \ref{rule:overlap} restricts the lower ellipse. See Figure \ref{fig:contours_mixed_constrained}.
    \item Temperature-constrained: $\beta$ is large enough such that Rule \ref{rule:temp} restricts both the upper and the lower ellipse. See Figure \ref{fig:contours_temperature_constrained}.
\end{itemize}

\begin{theorem}\label{thm:conductivity_coeffs_bound}
    There exist $\alpha_\mathrm{diag}(\zeta)$ and $\alpha_\mathrm{anti}(\zeta) > 0$ such that the Chebyshev coefficients $c_{k_1k_2}$ of $F_\zeta$ are bounded by
    \begin{equation}\label{eqn:conductivity_coeffs_bound}
        |c_{k_1,k_2}|
        \leq
        C(\zeta)\,
        \exp\bigl[
            -\alpha_\mathrm{diag}(\zeta) \, (k_1+k_2)
            -\alpha_\mathrm{anti}(\zeta) \, |k_1 - k_2|
        \bigr]
    \end{equation}
    for some $C(\zeta) < \infty$ independent of $k_1,k_2$.
    In the limit $\beta \to \infty$, $\omega,\eta \to 0$ with $|\omega| \lesssim \eta$, and assuming $E_F \in \,(-1,1)$,
    we have that
    \begin{align*}
        \alpha_\mathrm{diag}(\zeta) &=
        \begin{cases}
            \Theta\bigl(\eta\bigr) & \text{if $\zeta$ is relaxation- or mixed-constrained}, \\
            \Theta\bigl(\beta^{-1}\bigr) & \text{if $\zeta$ is temperature-constrained}, \quad \text{and} \\
        \end{cases}
        \\
        \alpha_\mathrm{anti}(\zeta) &=
        \begin{cases}
            \Theta\bigl(\eta^{1/2}\bigr) & \text{if $\zeta$ is relaxation-constrained}, \\
            {\mathcal{O}\bigl(\beta^{-1}\bigr)} & {\text{if $\zeta$ is mixed-constrained},} \\
            {0} & {\text{if $\zeta$ is temperature-constrained}} \\
        \end{cases}
    \end{align*}
    and
    \[
    \left\{
    \begin{aligned}
        &\beta \lesssim \eta^{-1/2} && \text{if $\zeta$ is relaxation-constrained}, \\
        \eta^{-1/2} \lesssim {} & \beta \lesssim \eta^{-1} && \text{if $\zeta$ is mixed-constrained}, \\
        \eta^{-1} \lesssim {} & \beta && \text{if $\zeta$ is temperature-constrained}. \\
    \end{aligned}
    \right.
    \]
\end{theorem}

A proof of Theorem \ref{thm:conductivity_coeffs_bound} and exact formulae for $\alpha_\mathrm{diag}(\zeta),$ and $\alpha_\mathrm{anti}(\zeta)$ are provided in Appendix \ref{sec:proofs_numerics}.
Figures \ref{fig:coeffs_relaxation} to \ref{fig:coeffs_temperature} show Chebyshev coefficients matching the predictions of Theorem \ref{thm:conductivity_coeffs_bound} perfectly.

\begin{figure}
    \centering
    \subfloat[$\beta = \frac{\pi}{5\sqrt{\eta}}$ (far relaxation)]{\label{fig:coeffs_far_relaxation}\includegraphics{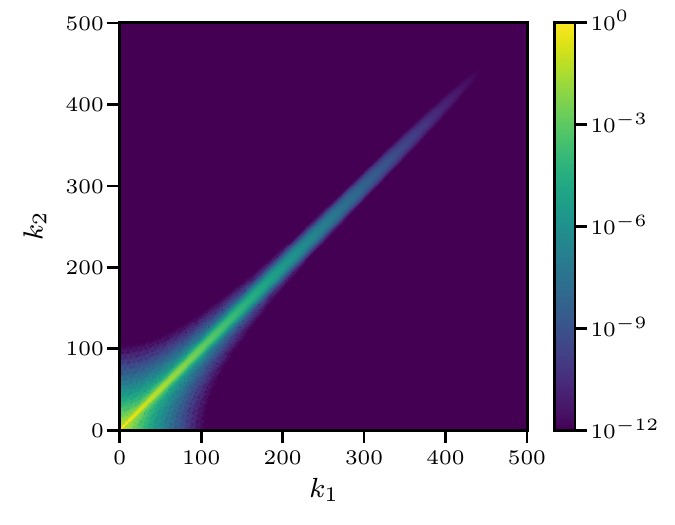}}
    \subfloat[$\beta = \frac{\pi}{\sqrt{\eta}}$ (relaxation)]{\label{fig:coeffs_relaxation}\includegraphics{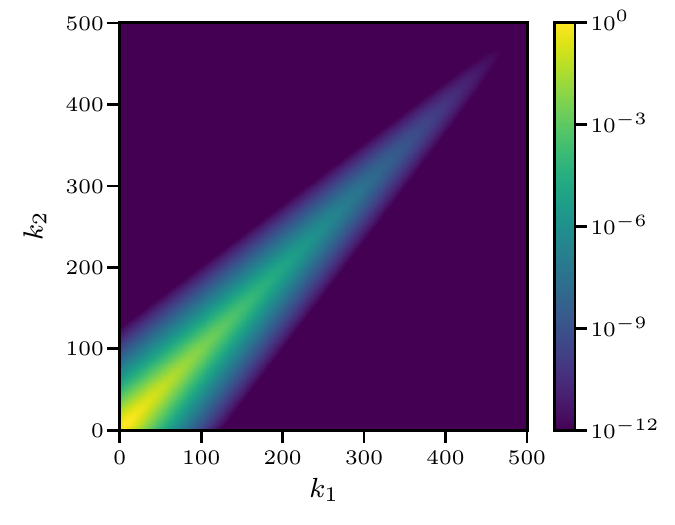}}

    \subfloat[$\beta = \frac{\pi}{2\eta}$ (mixed)]{\label{fig:coeffs_mixed}\includegraphics{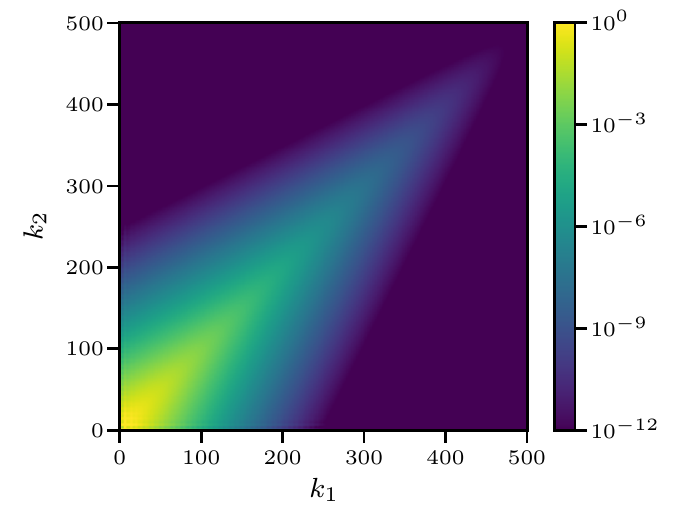}}
    \subfloat[$\beta = \frac{2\pi}{\eta}$ (temperature)]{\label{fig:coeffs_temperature}\includegraphics{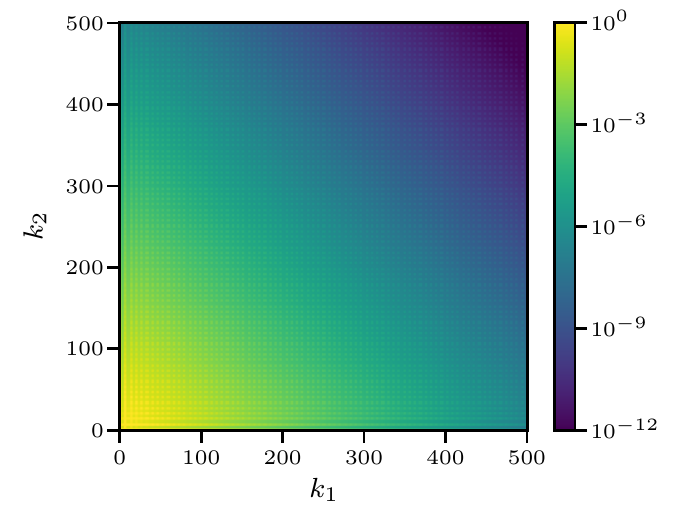}}

    \captionsetup{singlelinecheck=off}
    \caption[]{
        Normalized Chebyshev coefficients $\hat c_{k_1k_2} := {|c_{k_1k_2}|}/{|c_{00}|}$ of the conductivity function $F_\zeta$ with $E_F = \omega = 0$, $\eta = 0.06,$ and $\beta$ as indicated.
    }
    \label{fig:coeffs}
\end{figure}

We numerically observed the bound \eqref{eqn:conductivity_coeffs_bound} to describe the correct decay behavior and the decay rates of $\alpha_\mathrm{diag}(\zeta)$ and  $\alpha_\mathrm{anti}(\zeta)$ to be quantitatively accurate for temperature- and mixed-constrained parameters as well for relaxation-constrained parameters with $\beta$ close to the critical value $\beta \approx \eta^{-1/2}$.
For relaxation-constrained parameters far away from this critical value, however, the level lines of $c_{k_1k_2}$ are piecewise concave rather than piecewise straight as predicted by Theorem \ref{thm:conductivity_coeffs_bound}, see Figure \ref{fig:coeffs_far_relaxation}, and {we empirically found that} this extra concentration reduces the number of significant Chebyshev coefficients from $\mathcal{O}\bigl(\eta^{-3/2}\bigr)$ to $\mathcal{O}\bigl(\eta^{-1.1}\bigr)$, see Figure \ref{fig:ncoeffs}.

\begin{figure}
    \centering
    \includegraphics{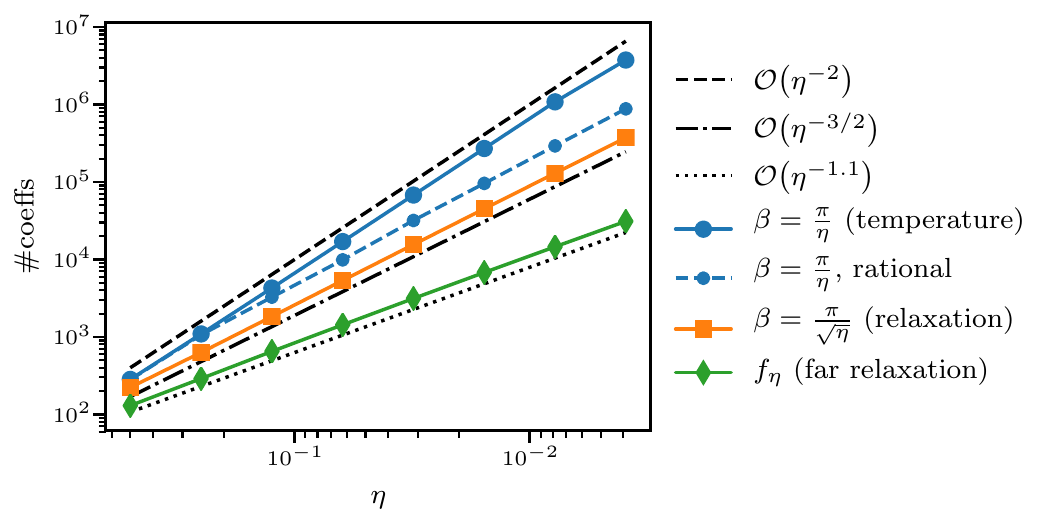}
    \caption{
    Number of normalized Chebyshev coefficients $\hat c_{k_1k_2} := {|c_{k_1k_2}|}/{|c_{00}|}$ larger than $10^{-3}$ for $F_\zeta$ with $E_F = \omega = 0$ and $f_\eta(E_1,E_2) := \frac{1}{E_1 - E_2 + \iota\eta}$.
    The ``rational'' line refers to the total number of Chebyshev coefficients in the pole expansion from Theorem~\ref{thm:pole_expansion} as described in Figure~\ref{fig:rational_ncoeffs}.
    }
    \label{fig:ncoeffs}
\end{figure}

Theorem \ref{thm:conductivity_coeffs_bound} suggests to truncate the Chebyshev series \eqref{eqn:chebyshev_series} using
\begin{equation}\label{eqn:truncated_indices}
    K(\tau) :=
    \Big\{
        (k_1,k_2) \in \mathbb{N}^2
        \mid
        \exp\bigl(
            -\alpha_\mathrm{diag} \, |k_1+k_2|
            -\alpha_\mathrm{anti} \, |k_1 - k_2|
        \bigr)
        \geq \tau
    \Big\},
\end{equation}
where here and in the following we no longer explicitly mention the dependence of $\alpha_\mathrm{diag}(\zeta), \alpha_\mathrm{anti}(\zeta)$ on $\zeta$.
The following theorem analyzes the error incurred by this approximation.

\begin{theorem}\label{thm:error_estimate}
    It holds that
    \begin{equation}\label{eqn:error_estimate}
        \big| \tilde \sigma_\ell^r[b] - \sigma_\ell^r[b] \big|
        =
        \mathcal{O}\Big(\alpha_\mathrm{diag}^{-1} \, \alpha_\mathrm{anti}^{-1} \, \tau \, |\log(\tau)|\Big)
        .
    \end{equation}
\end{theorem}
\begin{proof}
    See Appendix \ref{ssec:error_bound_proof}.
\end{proof}

In applications, we usually specify a truncation tolerance $\tau > 0$ such that \eqref{eqn:error_estimate} is upper-bounded by an error tolerance $\varepsilon > 0$.
It is shown in Appendix~\ref{ssec:tlogt_inverse} that this can be achieved by setting
$\tau_\varepsilon := \frac{\alpha_\mathrm{diag} \, \alpha_\mathrm{anti} \, \varepsilon}{|\log(\alpha_\mathrm{diag} \, \alpha_\mathrm{anti} \, \varepsilon)|}$, which yields
\begin{equation}\label{eqn:K_scaling}
|K(\tau_\varepsilon)|
=
\mathcal{O}\left(
    \frac{
        |\log(\alpha_\mathrm{diag} \, \alpha_\mathrm{anti} \, \varepsilon)|^2
    }{
        \alpha_\mathrm{diag} \, \alpha_\mathrm{anti}
    }
\right)
.
\end{equation}
Table \ref{tbl:Chebyshev_analysis} then follows by combining \eqref{eqn:K_scaling} with Theorem \ref{thm:conductivity_coeffs_bound}.

\subsection{Pole expansion for low-temperature calculations}\label{ssec:pole_expansion}

We have seen in the previous subsection that for increasing $\beta$, the sparsity in the Chebyshev coefficients of $F_\zeta$ induced by the factor $\frac{1}{E_1 - E_2 + \omega + \iota\eta}$ decreases and the number of coefficients eventually scales as $\mathcal{O}\bigl(\beta^2\bigr)$ such that Algorithm \ref{alg:local_conductivity} becomes expensive at low temperatures.
To avoid this poor low-temperature scaling, we propose to expand $F_\zeta$ into a sum over the poles in $S_\mathrm{temp}$ as described in Theorem \ref{thm:pole_expansion} below and apply Algorithm \ref{alg:local_conductivity} to each term separately.

\begin{theorem}\label{thm:pole_expansion}
    Let $k \in \mathbb{N}$ and denote by $\alpha_{k,\beta,E_F}$ the parameter of the ellipse through the Fermi-Dirac poles $E_F \pm \frac{(2k+1)\,\pi\iota}{\beta}$.
    There exists a function $R_{k,\beta,E_F}(E_1,E_2)$ analytic on the biellipse $E\bigl(\alpha_{k,\beta,E_F}\bigr)^2 \supset E\bigl(\alpha_{0,\beta,E_F}\bigr)^2$ such that
    \begin{equation}\label{eqn:pole_expansion}
        F_\zeta(E_1,E_2)
        =
        \tfrac{1}{E_1 - E_2 + \omega + \iota\eta}
        \left(
        \sum_{z \in Z_{k}} \tfrac{1}{\beta} \, \tfrac{1}{(E_1 - z) \, (E_2 - z)}
        +
        R_{k,\beta,E_F}(E_1,E_2)
        \right)\!
        ,
    \end{equation}
    where
    \[
    Z_{k}
    :=
    \bigl\{
    E_F + \tfrac{\ell \pi \iota}{\beta}
    \mid
    \ell \in \{ -2k+1, -2k+3, \ldots, 2k-3, 2k-1 \}
    \bigr\}
    \subset S_{\beta,E_F}
    .
    \]
\end{theorem}

\begin{proof}
    See Appendix \ref{ssec:pole_expansion_proof}.
\end{proof}

For $k$ large enough, the remainder term (the last term in \eqref{eqn:pole_expansion}) becomes relaxation-constrained and hence Algorithm \ref{alg:local_conductivity} becomes fairly efficient.
For the pole terms, on the other hand, we propose to employ Algorithm \ref{alg:local_conductivity} using the weighted Chebyshev approximation
\begin{equation}\label{eqn:pole_matching}
\frac{1}{(E_1 - z) \, (E_2 - z) \, (E_1 - E_2 + \omega + \iota\eta)}
\approx
\sum_{k_1k_2 \in K_z}
c(z)_{k_1k_2} \, \frac{T_{k_1}(E_1)}{E_1 - z} \, \frac{T_{k_2}(E_2)}{E_2 - z}
\end{equation}
where the weight $(E - z)^{-1}$ is chosen such that two factors $(E_1 - z)^{-1}$ and $(E_2 - z)^{-1}$ on the left- and right-hand side {match}.
The coefficients $c(z)_{k_1k_1}$ {in \eqref{eqn:pole_matching}} are therefore the Chebyshev coefficients of {the} relaxation-constrained function
\[
\frac{1}{E_1 - E_2 + \omega + \iota\eta}
\approx
\sum_{k_1k_2 \in K_z}
c(z)_{k_1k_2} \, T_{k_1}(E_1) \, T_{k_2}(E_2)
\]
and exhibit the concentration described in Theorem \ref{thm:conductivity_coeffs_bound}.
This leads us to the following algorithm.

\begin{algorithm}[H]
    \caption{Local conductivity via pole expansion}
    \label{alg:local_conductivity_via_poles}
    \begin{algorithmic}[1]
        \State\label{alg:local_conductivity_via_poles:R} $\tilde\sigma_\ell^r[b] := \int \frac{R_{k,\beta,E_F}(E_1,E_2)}{E_1 - E_2 + \omega + \iota\eta} \, d\mu^r_\ell(E_1,E_2)$, evaluated using Algorithm \ref{alg:local_conductivity}.
        \For{$z \in Z_{k,\beta,E_F}$}
            \State\label{alg:local_conductivity_via_poles:frac} $\tilde\sigma_\ell^r[b] := \tilde\sigma_\ell^r[b] + \frac{1}{\beta} \int \frac{1}{(E_1 - z) \, (E_2 - z) \, (E_1 - E_2 + \omega + \iota\eta)} \, d\mu^r_\ell(E_1,E_2)$,
            evaluated using
            \Statex \hspace{\algorithmicindent}
            Algorithm \ref{alg:local_conductivity} with the weighted Chebyshev polynomials $(E-z)^{-1} \, T_k(E)$.
        \EndFor
    \end{algorithmic}
\end{algorithm}

\begin{theorem}\label{thm:pole_expansion_cost}
    The dominant computational cost of Algorithm \ref{alg:local_conductivity_via_poles} is
    \begin{equation}\label{eqn:pole_expansion_cost}
        \#{\rm IP}
        =
        \mathcal{O}\bigl(k \, \eta^{-3/2}\bigr)
        +
        \begin{cases}
            \mathcal{O}\bigl(\eta^{-3/2}\bigr) & \text{if } \beta \, \eta^{1/2} \lesssim k, \\
            \mathcal{O}\bigl(\tfrac{\beta\eta^{-1}}{k}\bigr) & \text{if } \beta \eta \lesssim k \lesssim \beta \, \eta^{1/2}, \\
            \mathcal{O}\bigl(\tfrac{\beta^2}{k^2}\bigr) & \text{if } k \lesssim \beta \eta ,\\
        \end{cases}
    \end{equation}
    inner products if we assume that solving a single linear system of the form $(H - zI)^{-1} \, v$ is less expensive than $\mathcal{O}(\eta^{-3/2}\bigr)$ inner products (see Remark~\ref{rem:linear_solvers}).
    This cost is minimized if we choose
    \begin{equation}\label{eqn:optimal_k}
        k =
        \begin{cases}
            \Theta(1) & \text{if } \beta \lesssim \eta^{-1/2}, \\
            \Theta\bigl(\beta^{1/2} \, \eta^{1/4}\bigr) & \text{if } \eta^{-1/2} \lesssim \beta \lesssim \eta^{-3/2}, \\
            \Theta\bigl(\beta^{2/3} \, \eta^{1/2}\bigr) & \text{if } \eta^{-3/2} \lesssim \beta, \\
        \end{cases}
    \end{equation}
    which yields
    \begin{equation}\label{eqn:optimal_ip}
        \#{\rm IP}
        =
        \begin{cases}
            \mathcal{O}\bigl(\eta^{-3/2}\bigr) & \text{if } \beta \lesssim \eta^{-1/2}, \\
            \mathcal{O}\bigl(\beta^{1/2} \, \eta^{-5/4}\bigr) & \text{if } \eta^{-1/2} \lesssim \beta \lesssim \eta^{-3/2}, \\
            \mathcal{O}\bigl(\beta^{2/3} \, \eta^{-1}\bigr) & \text{if } \eta^{-3/2} \lesssim \beta. \\
        \end{cases}
    \end{equation}
\end{theorem}

\begin{proof}
    It follows from Theorem \ref{thm:conductivity_coeffs_bound} that the first term in \eqref{eqn:pole_expansion_cost} describes the cost of the for-loop in Algorithm \ref{alg:local_conductivity_via_poles} while the second term describes the cost of Line~\ref{alg:local_conductivity_via_poles:R}.
    Since the first term is strictly increasing while the second is decreasing, the sum of the two $\mathcal{O}$-terms is minimized by the unique $k$ such that the first term equals the second term, which one can readily verify to be given by \eqref{eqn:optimal_k}.
\end{proof}

We note that Algorithm \ref{alg:local_conductivity_via_poles} reduces to Algorithm \ref{alg:local_conductivity} if $\beta \lesssim \eta^{-1/2},$
but scales better than Algorithm \ref{alg:local_conductivity} for larger values of $\beta$,
e.g.,\ for $\beta \sim \eta^{-1} \sim \chi$ we have $\#{\rm IP} = \mathcal{O}\bigl(\chi^{7/4}\bigr)$ in the case of Algorithm \ref{alg:local_conductivity_via_poles}
while $\#{\rm IP} = \mathcal{O}\bigl(\chi^2\bigr)$ for Algorithm \ref{alg:local_conductivity}.
The first term in \eqref{eqn:pole_expansion} further reduces to $\mathcal{O}\bigl(k \, \eta^{-1.1}\bigr)$
if we assume the improved $\mathcal{O}\bigl(\eta^{-1.1}\bigr)$-scaling for the number of significant Chebyshev coefficients of $f(E_1,E_2) = \frac{1}{E_1 - E_2 + \omega + \iota\eta}$ suggested by Figure \ref{fig:ncoeffs}.
In this case, the optimal choice of $k$ and the corresponding costs are
\begin{equation}\label{eqn:pole_expansion_asymptotics}
k =
\begin{cases}
    \Theta(1) \\
    \Theta\bigl(\beta^{1/2} \, \eta^{0.05}\bigr) \\
    \Theta\bigl(\beta^{2/3} \, \eta^{0.37}\bigr) \\
\end{cases}
\,\,\text{and}\quad
\#{\rm IP} = \begin{cases}
\mathcal{O}\bigl(\eta^{-1.1}\bigr) & \text{if } \beta \lesssim \eta^{-1/2}, \\
\mathcal{O}\bigl(\beta^{1/2} \, \eta^{-1.05}\bigr) & \text{if } \eta^{-1/2} \lesssim \beta \lesssim \eta^{-3/2}, \\
\mathcal{O}\bigl(\beta^{2/3} \, \eta^{-0.73}\bigr) & \text{if } \eta^{-3/2} \lesssim \beta. \\
\end{cases}
\end{equation}
These predictions are compared against numerical results in Figure~\ref{fig:rational_ncoeffs} where we observe good qualitative agreement between the theory and the experiment.
For $\beta \sim \eta^{-1} \sim \chi$, equation \eqref{eqn:pole_expansion_asymptotics} yields $\#{\rm IP} = \mathcal{O}\bigl(\chi^{1.55}\bigr)$
which is only marginally more expensive than the $\mathcal{O}\bigl(\chi^{1.5}\bigr)$ cost of Algorithm~\ref{alg:local_conductivity} in the case of relaxation-constrained parameters $\beta^2 \sim \eta^{-1} \sim \chi$. This is empirically demonstrated by the ``rational'' line in Figure~\ref{fig:ncoeffs}.

\begin{figure}
    \subfloat[Number of coefficients]{\label{fig:ncoeffs_rational}\includegraphics{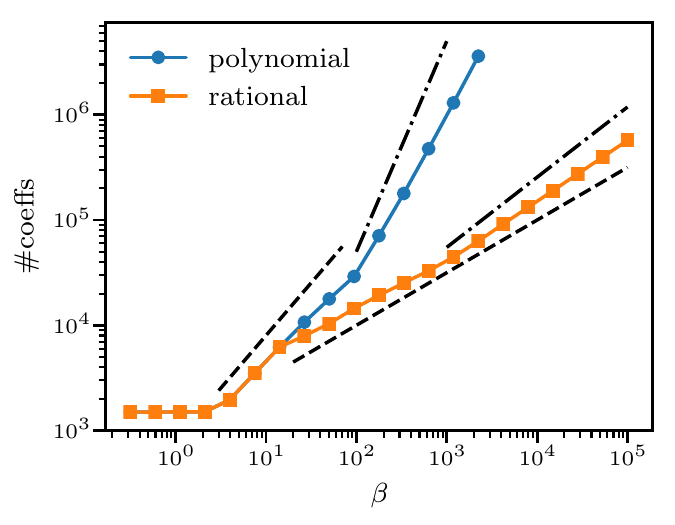}}
    \subfloat[Number of removed poles]{\label{fig:ncoeffs_rational_npoles}\includegraphics{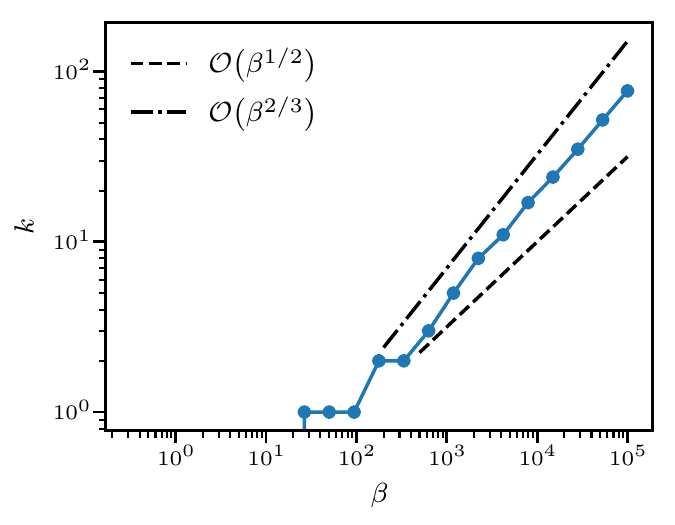}}
    \caption{
    \protect\subref{fig:ncoeffs_rational}
    Number of normalized Chebyshev coefficients $\hat c_{k_1k_2} := {|c_{k_1k_2}|}/{|c_{00}|}$ larger than $10^{-3}$ for $F_\zeta$ with $\eta = 0.06$ and $E_F = \omega = 0$.
    The ``polynomial'' line counts the number of significant coefficients in the Chebyshev expansion from \eqref{eqn:truncated_chebyshev_series}, while the ``rational'' line counts the sum of the number of significant Chebyshev coefficients of all the terms in the pole expansion from \eqref{eqn:pole_expansion}.
    The dashed lines denote $\mathcal{O}\bigl(\beta\bigr)$ and $\mathcal{O}\bigl(\beta^{1/2}\bigr)$, respectively, and the dash-dotted lines denote $\mathcal{O}\bigl(\beta^{2}\bigr)$ and $\mathcal{O}\bigl(\beta^{2/3}\bigr)$, respectively, cf.~\eqref{eqn:pole_expansion_asymptotics}.
    \protect\subref{fig:ncoeffs_rational_npoles}
    Index $k$ for the set of poles $Z_k$ from Theorem~\ref{thm:pole_expansion}.
    This number was determined by increasing $k$ starting from $0$ until the number of coefficients reported in \protect\subref{fig:ncoeffs_rational} stopped decreasing.
    }
    \label{fig:rational_ncoeffs}
\end{figure}

\begin{remark}\label{rem:pole_grouping}
    Instead of running Algorithm \ref{alg:local_conductivity} for each pole $z \in Z_{k,\beta,E_F}$ separately,
    we can apply Algorithm \ref{alg:local_conductivity} to a group of poles $\tilde Z \subset Z_{k,\beta,E_F}$ if we weigh the Chebyshev polynomials $T_k(E)$ with $q(E) := \prod_{z \in \tilde Z} (E-z)^{-1}$,
    and the same idea can also be used to improve the concentration of the Chebyshev coefficients of $R_{k,\beta,E_F}$.
    Grouping the poles in this manner reduces the computational cost of Algorithm \ref{alg:local_conductivity_via_poles}, but amplifies the round-off errors\footnote{
        We focus on rounding errors here for the sake of simplicity, but we will see in Subsection~\ref{ssec:pole_expansion_example} that a highly unbalanced $q$-factor also requires smaller approximation tolerances which in turn lead to larger runtimes.
    } by a factor $r := \max_{E \in [-1,1]} |q(E)| / \min_{E \in [-1,1]} |q(E)|$ such that the result is fully dominated by round-off errors if this ratio exceeds $10^{16}$.
    Since $|q(E_F)| \sim \beta^{|\tilde Z|}$ while $|q(\pm 1)| \sim 1$, this means that
    we have to keep the group size rather small (e.g.\ $|\tilde Z| \leq 4$ for $\beta = 10^4$) to maintain numerical stability.
    We therefore conclude that grouping poles reduces the prefactor, but does not change the asymptotics of the computational cost of Algorithm \ref{alg:local_conductivity_via_poles}.
\end{remark}

\begin{remark}\label{rem:flop_counts}
{
The runtime estimates \eqref{eqn:optimal_ip} and \eqref{eqn:pole_expansion_asymptotics} are formulated in terms of number of inner products and must therefore be multiplied by the length $|\Omega_r| = \mathcal{O}(r^2)$ of these inner products to obtain runtime estimates in terms of number of floating-point operations.
According to Theorem \ref{thm:main}, we must choose
}
\[
r
=
\begin{cases}
    \Theta\bigl(\eta^{-1}\bigr) & \text{if $\zeta$ is relaxation- or mixed-constrained}, \\
    \Theta\bigl(\beta\bigr) & \text{if $\zeta$ is temperature-constrained} \\
\end{cases}
\]
to guarantee an error in $\sigma_\ell^r[b]$ independent of $\zeta$; hence we conclude that Algorithm \ref{alg:local_conductivity_via_poles} requires
\begin{equation}\label{eqn:flop_estimate}
    {
    \left\{
    \begin{aligned}
        &\mathcal{O}\bigl(\eta^{-3.1}\bigr) && \text{if } \beta \lesssim \eta^{-1/2} \\
        &\mathcal{O}\bigl(\beta^{1/2} \, \eta^{-3.05}\bigr) && \text{if } \eta^{-1/2} \lesssim \beta \lesssim \eta^{-1} \\
        &\mathcal{O}\bigl(\beta^{5/2} \, \eta^{-1.05}\bigr) && \text{if } \eta^{-1} \lesssim \beta \lesssim \eta^{-3/2} \\
        &\mathcal{O}\bigl(\beta^{8/3} \, \eta^{-0.73}\bigr) && \text{if } \eta^{-3/2} \lesssim \beta \\
    \end{aligned}
    \right\}
    =
    \mathcal{O}\bigl((\beta+\eta^{-1})^{3.55}\bigr)
    }
\end{equation}
floating-point operations assuming the empirically observed scaling of the number of coefficients reported in \eqref{eqn:pole_expansion_asymptotics}.
In contrast, computing the eigendecomposition of $H_\mathrm{loc}$ and evaluating \eqref{e:local_conductivity} requires
\[
\mathcal{O}(|\omega_r|^3)
=
\mathcal{O}(r^6)
=
\begin{cases}
    \mathcal{O}\bigl(\eta^{-6}\bigr) & \text{if } \beta \lesssim \eta^{-1}, \\
    \mathcal{O}\bigl(\beta^6\bigr) & \text{if } \eta^{-1} \lesssim \beta \\
\end{cases}
\]
floating-point operations and hence scales with a power which is close to twice the one of our proposed algorithm.
\end{remark}

\begin{remark}\label{rem:linear_solvers}
    Solving a linear system $(H_\mathrm{loc} - z)^{-1} \, v$ associated with the two-dimensional configuration $\Omega_r$ using a direct solver takes
    \[
        \mathcal{O}\bigl( |\Omega_r|^{3/2}\bigr)
        =
        \mathcal{O}\bigl(r^3\bigr)
        =
        \begin{cases}
            \mathcal{O}\bigl(\eta^{-3}\bigr) & \text{if $\zeta$ is relaxation- or mixed-constrained}, \\
            \mathcal{O}\bigl(\beta^{3}\bigr) & \text{if $\zeta$ is temperature-constrained}, \\
        \end{cases}
    \]
    floating-point operations (see e.g.\ \cite[\S 7.6]{Dav06} regarding the runtime of direct sparse solvers).
    In comparison, approximating $p(E) \approx 1/(E - z)$ and evaluating $p(H_\mathrm{loc}) \approx (H_\mathrm{loc} - z)^{-1}$ (or equivalently, using an iterative linear solver like conjugate gradients) takes
    \[
    \mathcal{O}\bigl(\mathrm{degree}(p) \, |\Omega_r|\bigr)
    =
    \begin{cases}
        \mathcal{O}\bigl(\beta \, \eta^{-2}\bigr) & \text{if $\zeta$ is relaxation- or mixed-constrained}, \\
        \mathcal{O}\bigl(\beta^{3}\bigr) & \text{if $\zeta$ is temperature-constrained}, \\
    \end{cases}
    \]
    floating-point operations, where we used that $\mathrm{degree}(p) = \mathcal{O}\bigl(|\imag(z)|^{-1}\bigr) = \mathcal{O}\bigl(\beta\bigr)$ according to fundamental results in approximation theory, see e.g.\ \cite{Tre13}.
    We hence conclude that iterative solvers scale slightly better than direct ones in the relaxation- and mixed-constrained cases, and they scale as well as direct ones in the temperature-constrained case.
\end{remark}

\begin{remark}
    The cost of computing $\mathcal{O}\bigl(\eta^{-3/2}\bigr)$ inner products is
    \[
    \mathcal{O}\bigl(\eta^{-3/2} \, |\Omega_r|\bigr)
    =
    \begin{cases}
        \mathcal{O}\bigl(\eta^{-{7/2}}\bigr) & \text{if $\zeta$ is relaxation- or mixed-constrained}, \\
        \mathcal{O}\bigl(\eta^{-3/2} \, \beta^{2}\bigr) & \text{if $\zeta$ is temperature-constrained}, \\
    \end{cases}
    \]
    floating-point operations.
    Comparing this result against the findings of Remark \ref{rem:linear_solvers}, we conclude that the assumption in Theorem~\ref{thm:pole_expansion_cost} is satisfied if $\beta \lesssim \eta^{-3/2}$.
\end{remark}

\subsection{Remarks regarding implementation}\label{ssec:implementation}

We conclude this section by pointing out two features of the proposed algorithms which are relevant when one considers their practical implementation.

\subsubsection{Memory requirements}

Algorithm~\ref{alg:local_conductivity} as formulated above suggests that we precompute and store both the vectors $|v_{k_1}\rangle$ for all $k_1 \in K_1$ and $|w_{k_2}\rangle$ for all $k_2 \in K_2$.
This requires more memory than necessary since we can rewrite the algorithm as follows.

\begin{algorithm}[H]
    \caption{Memory-optimised version of Algorithm~\ref{alg:local_conductivity}}
    \label{alg:local_conductivity_memory}
    \begin{algorithmic}[1]
        \State Precompute $|v_{k_1}\rangle$ for all $k_1 \in K_1$ as in Algorithm~\ref{alg:local_conductivity}.
        \For{$k_2 \in K_2$ in ascending order}
            \State Evaluate $|w_{k_2}\rangle$ using the recurrence relation \eqref{eqn:chebrecursion}.
            \State Discard $|w_{k_2-2}\rangle$ as it will no longer be needed.
            \State Compute the inner products $\langle v_{k_1} | w_{k_2} \rangle$ for all $k_1$ such that $(k_1,k_2) \in K$, and
            \Statex \hspace{\algorithmicindent}
            accumulate the results as in Algorithm~\ref{alg:local_conductivity}.
        \EndFor
    \end{algorithmic}
\end{algorithm}

Furthermore, even caching all the vectors $|v_{k_1}\rangle$ is not needed if the function to be evaluated is relaxation-constrained: it follows from the wedge-like shape of the Chebyshev coefficients of such functions shown in Figure~\ref{fig:coeffs_relaxation} that in every iteration of the loop in Algorithm~\ref{alg:local_conductivity_memory}, we only need vectors $|v_{k_1}\rangle$ with index $k_1$ within some fixed distance from $k_2$.
The vectors $|v_{k_1}\rangle$ can hence be computed and discarded on the fly just like $|w_{k_2}\rangle$, albeit with a larger lag between computing and discarding.
Quantitatively, this reduces the memory requirements from $\mathcal{O}\bigl(\eta^{-1} \, |\Omega_r|\bigr)$ for both Algorithms \ref{alg:local_conductivity} and \ref{alg:local_conductivity_memory}
to $\mathcal{O}\bigl(\eta^{-1/2} \, |\Omega_r|\bigr)$ for the final version described above, assuming the function to be evaluated is relaxation-constrained.

\subsubsection{Choosing the approximation scheme}\label{ssec:approximation_scheme}

Algorithms~\ref{alg:local_conductivity} and \ref{alg:local_conductivity_via_poles} involve three basic operations, namely matrix-vector products, inner products and linear system solves, and a fundamental assumption in their derivation was that matrix-vector and inner products are approximately equally expensive and linear system solves are not significantly more expensive than that (see Theorem~\ref{thm:pole_expansion_cost} for the precise condition).
The former assumption is true in the sense that both matrix-vector and inner products scale linearly in the matrix size $m$, but their prefactors are very different: the inner product $\langle w \,|\, v \rangle$ takes $2m-1$ floating-point operations, while the cost of the matrix-vector product $H \, |v\rangle$ is approximately equal to twice the number of nonzeros in $H$.
Even in the simplest case of a single triangular lattice and a tight-binding Hamiltonian $H$ involving only nearest-neighbour terms and $s$ and $p$ orbitals, the number of nonzeros per column of $H$ is about 6 (number of neighbours) times 4 (number of orbitals), hence the cost of evaluating $H \,|v\rangle$ is approximately $48m$ which is 24 times more expensive than the inner product.
Similarly, the assumption regarding the costs of linear system solves holds true in the asymptotic sense as discussed in Remark~\ref{rem:linear_solvers}, but the situation may look very different once we include the prefactors.
This observation has two practical implications.
\begin{itemize}
    \item Rather than choosing the number of removed poles $k$ in Theorem~\ref{thm:pole_expansion} solely to minimise the number of coefficients, one should benchmark the runtimes of inner products, matrix-vector products and linear system solves and choose the $k$ which yields the smallest overall runtime.

    \item Fairly small values of $\eta$ are required before the wedge shown in Figure~\ref{fig:coeffs_relaxation} becomes thin enough that the savings due to a smaller number of inner products make a significant difference compared to the cost of the matrix-vector products, and {very large values of $\beta$ are required for the reduced number of inner products to compensate for the additional matrix-vector products and linear systems solves in Algorithm~\ref{alg:local_conductivity_via_poles}.}
\end{itemize}


\section{Numerical Demonstration}\label{sec:experiments}
This section demonstrates the theory developed in Sections \ref{sec:main} and \ref{sec:numerics} by applying it to a model bilayer system defined as follows.

\emph{Geometry.}
We consider a hexagonal bilayer system $\mathcal{R}_1 \cup \mathcal{R}_2$ with a relative twist angle of $2.5 ^\circ$ as shown in Figure \ref{fig:moire}.
The distance between the two layers is equal to the nearest-neighbor distance within each layer.
For ease of implementation, the projection onto a finite subsystem is performed using a parallelogrammatic cut-out
\[
\Omega_r
=
\mathop{\bigcup}_{\ell = 1}^2 \bigl\{A_\ell \, m : m \in \{-r, \ldots, r\}^2\bigr\}
\]
rather than the circular cut-out as in \eqref{e:Omega_r}.

\emph{Hamiltonian.}
We construct a model Hamiltonian $H$ for this system in two steps.
\begin{itemize}
    \item Define the matrix
    \begin{equation}\label{eqn:coupling_function}
        \tilde H_{R,R'} =
        h\bigl(|R - R'|\bigr) :=
        \begin{cases}
            \exp\left(-\tfrac{|R - R'|^2}{r_\mathrm{cut}^2 - |R - R'|^2}\right) & \text{if } |R - R'| < r_\mathrm{cut}, \\
            0                                                                   & \text{otherwise},
        \end{cases}
    \end{equation}
    where $R$ and $R'$ range over all lattice sites in $\mathcal{R}_1 \cup \mathcal{R}_2$ and
    \[
    r_\mathrm{cut} = \sqrt{3} \times \text{(nearest-neighbor distance)}
    \]
    denotes the second-nearest-neighbor distance in the lattices.
    Note that this implies that if $R,R'$ are sites on the same lattice, then
    \[
    \tilde H_{R,R'} \neq 0
    \quad\iff\quad
    R = R' \text{ or } R,R' \text{ are nearest neighbors}
    .
    \]

    \item Set $H$ to be a shifted and scaled copy of $\tilde H$ such that the spectrum of $H$ is contained in $[-1,1]$, i.e.\
    \[
    H
    =
    \tfrac{2}{\tilde E_\mathrm{max} - \tilde E_\mathrm{min}}
    \,
    \Bigl( \tilde H - \tfrac{\tilde E_\mathrm{max} + \tilde E_\mathrm{min}}{2} \, I \Bigr)
    \]
    where $\tilde E_\mathrm{min}$ and $\tilde E_\mathrm{max}$ denote lower and upper bounds, respectively, on the spectrum of $\tilde H$.
\end{itemize}

All numerical experiments in this section have been performed on a single core of an Intel Core i7-8550 CPU (1.8 GHz base frequency, 4 GHz turbo boost) using the Julia programming language \cite{BEKS17}.

\subsection{Convergence with respect to the localization radius $r$}\label{ssec:r_convergence}

\begin{figure}
    \includegraphics{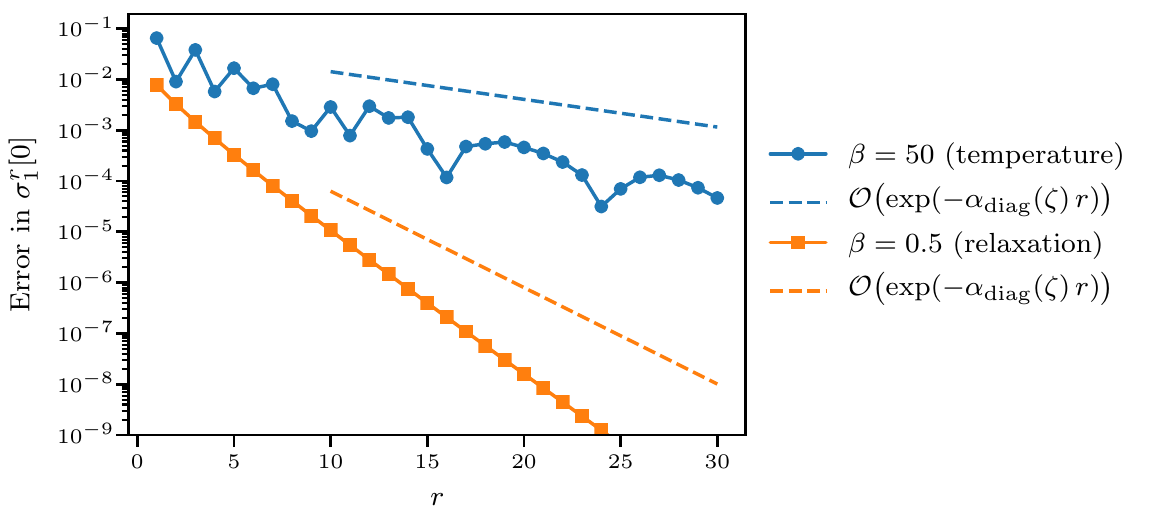}
    \caption{
    Convergence of $\sigma_1^r[0]$ as a function of $r$ for $E_F = \omega = 0$, $\eta = 0.5$ and $\beta$ as indicated.
    Errors were measured by comparing against the result for $r = 50$.
    }
    \label{fig:r_convergence}
\end{figure}

We have seen in Theorem \ref{thm:main} that the local conductivity $\sigma_\ell^r[b]$ converges exponentially, as $r \to \infty$, with exponent proportional to $\min\{\beta^{-1},\eta\}$.
Since the particular Hamiltonian $H$ we consider involves only nearest-neighbor interactions, this statement can be further sharpened. The approximate local conductivity $\tilde \sigma_\ell^r[b]$ introduced in \eqref{eqn:approximate_local_conductivity} is now independent of $r$, as long as
\[
r \geq \max_{(k_1,k_2) \in K} \tfrac{1}{2} \, (k_1 + k_2 + 2).
\]
Hence, $\tilde \sigma_\ell^r[b]$ equals the exact local conductivity $\sigma_\ell^\infty[b]$ in the thermodynamic limit $r \to \infty$ up to truncation of the Chebyshev series.
Combining this observation with the decay of the Chebyshev coefficients of $F_\zeta$, asserted in Theorem \ref{thm:conductivity_coeffs_bound}, yields
\begin{equation}\label{eqn:local_conductivity_convergence}
\bigl|\sigma_\ell^r[b] - \sigma_\ell^\infty[b]\bigr|
\leq
C \, \exp\bigl( - \alpha_\mathrm{diag}(\zeta) \, r\bigr)
\end{equation}
for some $C > 0$ independent of $r$.
This theoretical finding is numerically confirmed in Figure \ref{fig:r_convergence}, which demonstrates that $\sigma_\ell^r[b]$ indeed converges exponentially with a rate of convergence upper bound by $\alpha_\mathrm{diag}(\zeta)$ with reasonable but not perfect tightness.

The above argument for relating localization to polynomial approximation is based on closely related arguments from \cite{BBR13,DMS84}.

\subsection{Scaling for relaxation-constrained parameters}\label{ssec:eta_scaling}

\begin{figure}
    \includegraphics{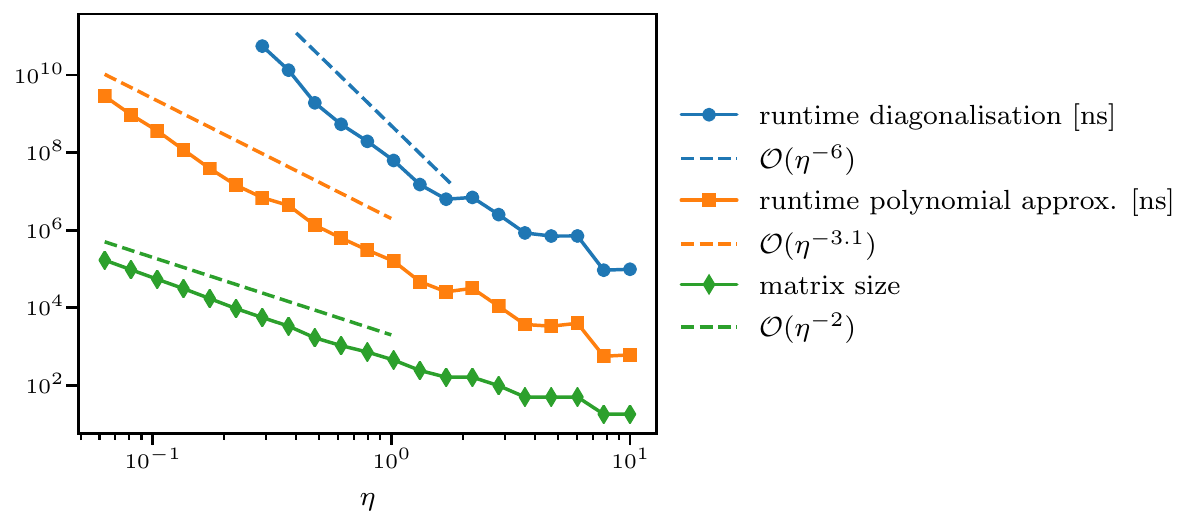}
    \caption{
        Runtime of local conductivity calculations via diagonalisation of $H$ and polynomial approximation of $F_\zeta$ (Algorithm \ref{alg:local_conductivity}), respectively, for $\beta = 0.1$, $E_F = \omega = 0$ and varying $\eta$.
        The truncated Chebyshev expansion $\tilde F_\zeta(E_1,E_2)$ in \eqref{eqn:truncated_chebyshev_series} has been determined by computing all Chebyshev coefficients $c_{k_1k_2}$ for $k_1,k_2 \in \{0, \ldots 500\}$ and then dropping the coefficients of smallest absolute values until the sum of the dropped coefficients reaches $10^{-3}$.
        The matrix size $|\Omega_r|$ is determined by choosing $r = \max_{(k_1,k_2) \in K} \tfrac{1}{2} \, (k_1 + k_2 + 2)$, cf.\ Subsection~\ref{ssec:eta_scaling}.
    }
    \label{fig:eta_scaling}
\end{figure}

The discussion in Subsection \ref{ssec:r_convergence} suggests to choose the localization radius $r$ by determining a truncated Chebyshev series approximation $\tilde F_\zeta$ of sufficient accuracy and then setting
\[
r = \max_{(k_1,k_2) \in K} \tfrac{1}{2} \, (k_1 + k_2 + 2)
\]
where $K$ denotes the set of indices in $\tilde F_\zeta$, cf.\ \eqref{eqn:truncated_chebyshev_series}.
Figure \ref{fig:eta_scaling} demonstrates that this choice of $r$ leads to fairly large matrix sizes $|\Omega_r|$ and hence the diagonalisation algorithm is not competitive with our Algorithm~\ref{alg:local_conductivity} for any of the parameters $\zeta$ considered in Figure \ref{fig:eta_scaling}.
However, we remark that unlike diagonalisation, Algorithm~\ref{alg:local_conductivity} benefits from the excellent sparsity of the Hamiltonian $H$ and the relaxed error tolerance $\varepsilon = 10^{-3}$ considered in this example. The relative performance of Algorithm~\ref{alg:local_conductivity} may therefore be somewhat worse for more realistic Hamiltonians.

\subsection{Pole expansion for temperature-constrained parameters}\label{ssec:pole_expansion_example}

\def\TT{\rule{0pt}{2.6ex}}       
\def\BB{\rule[-1.2ex]{0pt}{0pt}} 

\begin{table}
    \subfloat[$\beta = 20$, $|\Omega_r| = 13\,122$]{
    \begin{tabular}{|c|c|c|c|c|}
        \hline
        \multicolumn{2}{|c|}{} & polynomial & pole expansion & grouped pole expansion
        \TT\BB \\ \hline\hline
        \multirow{2}{*}{matvec} &
        count & 225 & 509 & 396
        \TT \\
        & time [s] & 0.056 & 0.155 & 0.124
        \BB \\ \hline
        \multirow{2}{*}{inner prod} &
        count & 2680 & 602 & 229
        \TT \\
        & time [s] & 0.015 & 0.003 & 0.001
        \BB \\ \hline\hline
        \multicolumn{2}{|r|}{Total time [s]}
        & 0.072 & 0.159 & 0.125
        \TT\BB \\ \hline
    \end{tabular}
    }

    \subfloat[$\beta = 30$, $|\Omega_r| = 30\,258$]{
    \begin{tabular}{|c|c|c|c|c|}
        \hline
        \multicolumn{2}{|c|}{} & polynomial & pole expansion & grouped pole expansion
        \TT\BB \\ \hline\hline
        \multirow{2}{*}{matvec} &
        count & 348 & 772 & 741
        \TT \\
        & time [s] & 0.338 & 0.749 & 0.798
        \BB \\ \hline
        \multirow{2}{*}{inner prod} &
        count & 6410 & 739 & 468
        \TT \\
        & time [s] & 0.182 & 0.014 & 0.007
        \BB \\ \hline\hline
        \multicolumn{2}{|r|}{Total time [s]}
        & 0.520 & 0.763 & 0.806
        \TT\BB \\ \hline
    \end{tabular}
    }

    \vspace{0.5em}

    \caption{
    Runtimes of Algorithm~\ref{alg:local_conductivity} (polynomial approximation), Algorithm~\ref{alg:local_conductivity_via_poles} (pole expansion) and Algorithm~\ref{alg:local_conductivity_via_poles} with all poles grouped into a single term as described in Remark~\ref{rem:pole_grouping}, for $\beta$ as indicated, $E_F = -0.2$, $\eta = 1$, $\omega = 0$ and number of removed poles $k = 3$.
    The matrix sizes $|\Omega_r|$ have been determined as in Figure~\ref{fig:eta_scaling}.
    All linear system solves $\bigl(\prod_{k} (H_\mathrm{loc} - z_k)^{-1} \bigr)\, v$ have been perform using polynomial approximation (cf.\ Remark~\ref{rem:linear_solvers}), and the corresponding matrix-vector products are included in the matvec count reported above.
    }
    \label{tbl:pole_expansion}
\end{table}

Table~\ref{tbl:pole_expansion} demonstrates the effect of accelerating the polynomial-approximation-based Algorithm~\ref{alg:local_conductivity} using pole expansion as described in Subsection~\ref{ssec:pole_expansion}.
We observe the following.
\begin{itemize}
    \item The additive approximation scheme described in Algorithm~\ref{alg:local_conductivity_via_poles} significantly reduces the inner products count compared to the polynomial algorithm, and grouping poles as described in Remark~\ref{rem:pole_grouping} reduces the inner product count even further.

    \item The runtimes of all three algorithms are dominated by the matrix-vector (matvec) products.
    The matvec counts are significantly larger for the two rational algorithms; hence their overall runtimes are larger than that of the polynomial algorithm.
\end{itemize}
The larger number of matrix-vector products in the rational algorithms is due to several factors.
\begin{enumerate}
    \item Pole expansion without grouping (Algorithm~\ref{alg:local_conductivity_via_poles}) requires running Algorithm~\ref{alg:local_conductivity} multiple times and hence incurs more matrix-vector products from Lines \ref{alg:local_conductivity:v} and \ref{alg:local_conductivity:w} of Algorithm~\ref{alg:local_conductivity}.

    \item\label{itm:multiple_vw}The rational algorithms require solving sequences of linear systems $\bigl(\prod_k (H - z_k)^{-1} \bigr) \, v$, which we evaluate by approximating $q(E) \approx \prod_k (E - z_k)^{-1}$ and replacing $\bigl(\prod_k (H - z_k)^{-1} \bigr) \, v \to q(H) \, v$.

    \item\label{itm:large_q}Determining polynomials $p(E_1,E_2)$ and $q(E)$ such that
    \[
        p(E_1,E_2) \, q(E_1) \, q(E_2) \approx F_\zeta(E_1,E_2)
    \]
    requires stricter tolerances and hence larger degrees due to the multiplications (cf.\ Remark~\ref{rem:pole_grouping}).
\end{enumerate}
Item~\ref{itm:multiple_vw} explains why the matvec count is higher for the ungrouped pole expansion compared to the grouped pole expansion for $\beta = 20$, while Item~\ref{itm:large_q} explains why the matvec count for grouped pole expansion catches up with that of ungrouped pole expansion for larger values of $\beta$ where the ratio $\bigl(\max_{E} q(E)\bigr) / \bigl(\min_{E} q(E)\bigr)$ is larger.

These findings suggest that the rational approximation techniques from Subsection~\ref{ssec:pole_expansion} require very large values of $\beta$ to outperform the polynomial algorithm from Subsection~\ref{ssec:algorithm}.
However, we also note that the performance of the rational algorithms can be improved by using better approximation and evaluation schemes.

\subsection{Convergence of integral over configurations}\label{ssec:quad_convergence}

\begin{figure}
    \includegraphics{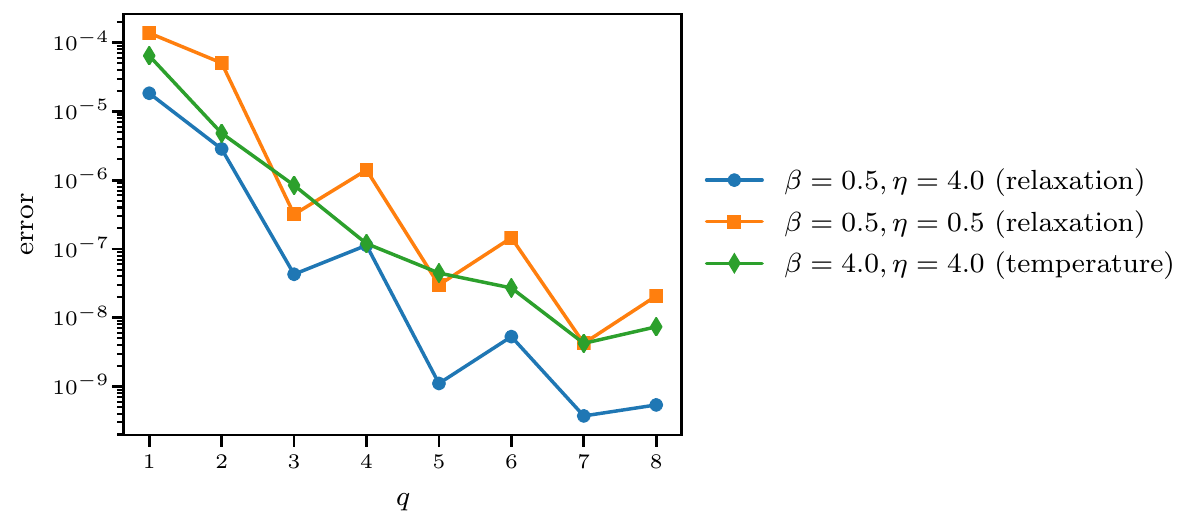}
    \caption{
        Convergence of the $q^2$-point bivariate trapezoidal rule applied to the integral over configurations in \eqref{e:conductivity_integral} for $E_F = \omega = 0$ and $\beta$, $\eta$ as indicated.
        Errors were computed relative to the result for $q = 10$.
    }
    \label{fig:quad_convergence}
\end{figure}

Finally, we demonstrate in Figure~\ref{fig:quad_convergence} the convergence of the periodic bivariate trapezoidal rule applied to the integral over configurations in \eqref{e:conductivity_integral}. We observe the following.
\begin{itemize}
    \item The coupling function $h(r)$ introduced in \eqref{eqn:coupling_function} is $C^\infty$ but not analytic, which according to Lemma~\ref{thm:local} implies that also the local conductivity $\sigma_\ell^r[b]$ as a function of the configurations $b$ is $C^\infty$ but not analytic.
    We therefore expect that trapezoidal rule quadrature applied to this function converges superalgebraically but subexponentially, and this matches our numerical observations in Figure~\ref{fig:quad_convergence}.

    \item Conversely to Figure~\ref{fig:r_convergence}, the convergence with respect to $q$ is fairly monotonous for relaxation-constrained parameters but oscillates for temperature-constrained parameters.
\end{itemize}

\section{Conclusion} \label{sec:conclusion}
We have demonstrated in this paper how to construct numerical algorithms for conductivity in incommensurate heterostructures where classical Bloch theory is unavailable.
Our construction is based on the observation that the ergodicity property of incommensurate bilayers allows us to replace conductivity calculations on the infinite system with an integral over the two unit cells.
The resulting formula presented in Section~\ref{sec:main} is similar to Bloch's theorem and extends an analogous construction for the density of states in \cite{massatt2017}.
Unlike in Bloch's theorem, however, the two unit cells require padding with a buffer region which may involve tens of thousands of atoms. This is far beyond the reach of the diagonalization algorithm; hence we propose in Section~\ref{sec:numerics} an alternative, linearly scaling algorithm in the spirit of the Kernel Polynomial Method and Fermi Operator Expansion. We show that for relaxation-constrained parameters $\beta \lesssim \eta^{-1/2}$, our algorithm requires only $\mathcal{O}\bigl(\eta^{-3/2}\bigr)$ inner products, and we present a rational approximation scheme which effectively allows us to reduce arbitrary parameter regimes to the relaxation-constrained case.

\section*{Acknowledgement}
The authors would like to thank Stephen Carr and Paul Cazeaux for helpful comments on
the theme of this paper.

\appendix


\section{Proofs: Conductivity}

\subsection{Notation}
Throughout several of the following proofs it will become necessary to compare
resolvent matrices $(z - H^r)^{-1}$ and $(z - H^{r'})^{-1}$ of different size $r, r'$. To
that end, it is convenient to implicitly extend all matrices to be defined over
$\Omega$. Specifically: if $A$ is usually defined over $\Omega_r$, then we use
the implicit extension to $\Omega$ given by
\[ [ A]_{R\alpha,R'\alpha'} =
  \begin{cases}
    A_{R\alpha,R'\alpha'},       & \quad \text{if } R\alpha \in \Omega_{r}, R'\alpha' \in \Omega_{r},\\
   0,			& \quad \text{otherwise}.\\
  \end{cases}
\]

\subsection{Proof of Lemma \ref{thm:local}}
We let $\Lambda = [-1,1]$ and recall that this interval contains the
spectrum for all Hamiltonians $H^r$, $r > 0$.
Letting $r>0$ and $a > 0,$ then following the same argument as \cite[Lemma 4.2]{massatt2017} we have the existence of $\tilde \gamma > 0$ such that, for $z \in \mathbb{C}$ with $d(z, \Lambda) > a/2$,
and $\Omega' \subset \Omega$ such that $\Omega_r \subset \Omega'$,
\begin{equation}
\label{e:bound}
\begin{split}
\Big|\big[(z-H^r_\ell(b))^{-1}\big]_{R\alpha,R'\alpha'}& - \big[(z-H_\ell(b)|_{\Omega'})^{-1}\big]_{R\alpha,R'\alpha'}\Big| \\
&\lesssim a^{-6} \min\Big\{ e^{-a \tilde \gamma |R-R'|}, e^{-a\tilde  \gamma (r - \max\{|R|,|R'|\})}\Big\}.
\end{split}
\end{equation}
We have the following Lemma:
\begin{lemma}
\label{lemma:thermo}
Using Assumption \ref{assump:decay}, we have
\begin{equation*}
{(z-H_\ell(b))^{-1} = \lim_{r\rightarrow\infty}(z-H_\ell^r(b))^{-1}}.
\end{equation*}
Further, $(z-H_\ell(b))^{-1}$ is periodic over $\Gamma_{\tau(\ell)}$.
\end{lemma}
\begin{proof}
{From (\ref{e:bound}), we have that $(z-H_\ell^r(b))^{-1}$ is Cauchy over $\mathcal{L}(\ell^2(\Omega))$, and hence has a well defined limit. This limit must be $(z-H_\ell(b))^{-1}$ as it is clearly true on the dense subset of vectors with a finite number of entries. } $ \|H_\ell^r(b + 2\pi A_{\tau(\ell)} n) - H_\ell^r(b)\|_{\rm op} \rightarrow 0$ for $n \in \mathbb{Z}^2$ as $r \rightarrow \infty$, and hence $(z- H_\ell(b))^{-1}$ is periodic over $\Gamma_{\tau(\ell)}.$
\end{proof}
Let $P_s : \ell^2(\Omega) \rightarrow \ell^2(\Omega)$ be the projection defined by
\begin{equation*}
[P_s \psi]_{R\alpha} = \delta_{|R| < s}\psi_{R\alpha}.
\end{equation*}

We now introduce two lemmas we will use for the convergence estimates. The
matrix $A$ in Lemma \ref{lemma:inside} corresponds to resolvent differences as
in \eqref{e:bound}, while the second lemma will be applied to
resolvents and localized Hamiltonian operators.
\begin{lemma}
\label{lemma:inside}
For $A \in \mathcal{L}(\ell^2(\Omega))$ satisfying (for $r > 1$) 
\begin{equation*}
|A_{R\alpha,R'\alpha'}| \lesssim e^{- \tilde c\log(a)}\min\{e^{-a \gamma_c |R-R'| },e^{-a \gamma_c(r-\max\{|R|,|R'|\})}\},
\end{equation*}
 it holds that{
\begin{equation*}
\|P_{r/2}AP_{r/2}\|_{\rm op} \lesssim e^{-\gamma_ca r/2 -  c\log(a)+c\log(r)}.
\end{equation*}}
\end{lemma}
\begin{proof}
 We estimate
\begin{equation*}
\begin{split}
\|P_{r/2}AP_{r/2}\|_{\rm op}^2 &\leq  \|P_{r/2}AP_{r/2}\|_F^2\\
&\lesssim e^{-\gamma_c ar - 2\tilde c \log(a) } |\Omega_{r/2}|^2 \\
& \lesssim r^4 e^{-\gamma_c ar - 2\tilde c \log(a) } \\
&  \lesssim  e^{-\gamma_ca r - 2\tilde c  \log(a) + 4\log(r)}
\end{split}
\end{equation*}
for  $c = \max\{2\tilde c,4\} $, so we then have
\begin{equation*}
\|P_{r/2}AP_{r/2}\|_{\rm op} \lesssim e^{-\gamma_ca r/2 -  c\log(a)+c\log(r)}.
\end{equation*}
\end{proof}

Recall $e_{0\alpha} \in \ell^2(\Omega)$ such that $[e_{0\alpha}]_{R\alpha'} =
\delta_{0R}\delta_{\alpha\alpha'}$.

\begin{lemma}
\label{lemma:product_and_localize}
If $A, A^{(1)}, A^{(2)} \in \mathcal{L}(\ell^2(\Omega))$ satisfies
\begin{equation*}
|A^{(j)}_{R\alpha,R'\alpha'}| \lesssim e^{-\gamma_c a |R-R'| - \tilde c \log(a)}
\end{equation*}
for some $\gamma_c > 0$, then there exist $\gamma_d, c > 0$ such that
\begin{align}
   \label{eq:localize}
   {\||(1-P_{r/2}) A |e_{0\alpha}\rangle\|_{\ell^2}}
   &\lesssim e^{- \gamma_d a r -  c \log(a)} \qquad \text{and} \\
   \label{eq:product}
   \big|[A^{(1)}A^{(2)}]_{R\alpha,R'\alpha'}\big| &\lesssim e^{-\gamma_d a |R-R'|- c \log(a)}.
\end{align}
\end{lemma}
\begin{proof}
   The two estimates result follows from straightforward direct estimations of
   the individual vector or matrix entries of, respectively, $A e_{0\alpha}$ and
   $[A^{(1)}A^{(2)}]_{R\alpha,R'\alpha'}$.
\end{proof}

To proceed with the proof of Lemma \ref{thm:local}, we recognize that we can
rewrite the current-current correlation measure in terms of a contour integral.

\begin{lemma} \label{th:contour}
   Let $\phi$ be analytic on $S_{a}\times S_{a}$ and $\C_a \subset S_{a} -
   S_{a/2}$ a complex contour encircling the spectrum of
   $H^r_\ell(b)$, then
   \begin{equation}
   \label{e:resolvent_form}
   \begin{split}
   & \int_{\mathbb{R}^2} \phi(E_1,E_2) d\mu_\ell^r[b](E_1,E_2) \\
   &\qquad = -\frac{1}{4\pi^2} \oint_{z'\in\C_a}\oint_{z \in \C_a} \phi(z,z')
   \\  &\qquad\qquad\quad \sum_{\alpha \in \A_\ell}\langle e_{0\alpha}|(z-H^r_\ell(b))^{-1} \partial_p H^r_\ell(b) (z'-H^r_\ell(b))^{-1}\partial_{p'} H^r_\ell(b)|e_{0\alpha}\rangle dzdz'.
   \end{split}
   \end{equation}
\end{lemma}
\begin{proof}
   Inserting the spectral decomposition of $H^r_\ell(b)$ into the right-hand
   side of \eqref{e:resolvent_form} and then applying Cauchy's integral formula
   twice yields the definition \eqref{e:finite} of the local current-current
   correlation measure $\mu_\ell^r[b]$.
\end{proof}

For the remainder of this proof, we denote $P = P_{r/2}$ for the sake of
brevity. Then,
\begin{equation*}
\begin{split}
\langle e_{0\alpha}|&(z-H^r_\ell(b))^{-1} \partial_p H^r_\ell(b) (z'-H^r_\ell(b))^{-1}\partial_{p'} H^r_\ell(b)|e_{0\alpha}\rangle \\
& = \sum_{U_i \in \{P,1-P\} }\langle e_{0\alpha}|(z-H^r_\ell(b))^{-1}U_1\partial_p H^r_\ell(b) U_2 (z'-H^r_\ell(b))^{-1}U_3\partial_{p'} H^r_\ell(b)|e_{0\alpha}\rangle \\
&= S_1^r + S_2^r,
\end{split}
\end{equation*}
where $S_1^r = S_1^r(z,z')$, $S_2^r = S_2^r(z,z')$ are given by
\begin{align*}
S_1^r &=\langle e_{0\alpha}|(z-H^r_\ell(b))^{-1}P \partial_p H^r_\ell(b) P(z'-H^r_\ell(b))^{-1}P\partial_{p'} H^r_\ell(b)|e_{0\alpha}\rangle  \qquad \text{and} \\
S_2^r &= \hspace{-6mm}
   \sum_{\substack{U_i \in \{P,1-P\} \\ (U_1,U_2,U_3) \neq (P,P,P)}}
      \hspace{-6mm}
   \langle e_{0\alpha}|(z-H^r_\ell(b))^{-1}U_1 \partial_p H^r_\ell(b) U_2 (z'-H^r_\ell(b))^{-1}U_3\partial_{p'} H^r_\ell(b)|e_{0\alpha}\rangle.
\end{align*}
Using Lemma \ref{lemma:thermo} and the resolvent formulation above, we can see that the {weak limit $\mu_\ell[b] := \lim_{r\rightarrow\infty}\mu_\ell^r[b]$  and the limit $S_j := \lim_{r\rightarrow\infty}S_j^r$ exist}. However, we wish to obtain an error estimate.
We can estimate
\begin{align}
   \notag
\biggl|\int_{\mathbb{R}^2}&F(E_1,E_2)d\mu_\ell^r[b](E_1,E_2)  - \int_{\mathbb{R}^2}F(E_1,E_2)d\mu_\ell[b](E_1,E_2)\biggr| \\
\notag
& \lesssim \oint_{z'\in\C_a}\oint_{z \in \C_a} |F(z,z')| \biggl|S_1^r+S_2^r - S_1-S_2\biggr|dzdz' \\
\notag
& \lesssim \sup_{z,z' \in \C_{a}}  |F(z,z')| \cdot \sup_{z,z' \in \C_a}  \big|S_1^r+S_2^r - S_1-S_2\big| \\
\label{eq:contour-to-DeltaS1r}
&\leq \sup_{z,z' \in \C_a}  |F(z,z')| \cdot \sup_{z,z' \in \C_a}  \big(|S_1^r-S_1| + |S_2^{r}|+|S_2|\big).
\end{align}
Applying Lemma \ref{lemma:product_and_localize}, we readily obtain
\begin{equation} \label{eq:estimate_S2r}
|S_2^r| \lesssim e^{-\gamma_a r a - c' \log(a)}
\end{equation}
for some constants $\gamma_a, c' >0$.

Next, we claim that there exist constants $\gamma_b, c''$ such that
\begin{equation} \label{eq:S1r-S1rprime}
   |S_1^r-S_1|
   \lesssim e^{-\gamma_b r a - c'' \log(a)+c''\log(r)}.
\end{equation}

\begin{proof}[Proof of \eqref{eq:S1r-S1rprime}]
   We define two sets of operators,
   \begin{align*}
   \Delta\mathcal{B}_{r} = \bigl\{\,&  P[(z-H^r_\ell(b))^{-1}-(z-H_\ell(b))^{-1}]P, \\
   &P[\partial_pH^r_\ell(b)-\partial_pH_\ell(b) ]P,\\
   & P[(z'-H^r_\ell(b))^{-1}-(z'-H_\ell(b))^{-1}]P,\\
   & P[\partial_p H^r_\ell(b)-\partial_{p'} H_\ell(b)]P\bigr\},
      \qquad \text{and} \\[2mm]
      \mathcal{B}_{r} = \bigl\{ \,&  P(z-H^r_\ell(b))^{-1}P, P\partial_pH^r_\ell(b)P, P(z'-H^r_\ell(b))^{-1}P, P\partial_{p'} H^r_\ell(b)P \\
      & P(z-H_\ell(b))^{-1}P, P\partial_pH_\ell(b)P, P (z'-H_\ell(b))^{-1}P, P\partial_{p'} H_\ell(b)P\bigr\}.
   \end{align*}
   Then, we can decompose
   \begin{equation} \label{eq:S1r-S1rprime-decompose}
      S_1^r-S_1 = \sum_j \langle e_{0\alpha}|A_1^{(j)} A_2^{(j)} A_3^{(j)} A_4^{(j)} |e_{0\alpha}\rangle,
   \end{equation}
   where each of the operators
   $A_i^{(j)} \in \mathcal{B}_{r} \cup \Delta \mathcal{B}_{r}$ and
   for every $j$ at least one $A_i^{(j)} \in \Delta\mathcal{B}_{r}$.

   Using Lemma \ref{lemma:inside}, it is straightforward to see that
   \begin{align*}
      \|A\|_{\rm op} &\lesssim \max\{a^{-1},1\} \qquad \text{for $A \in \mathcal{B}_{r}$ and} \\
      \|A\|_{\rm op} &\lesssim e^{-\gamma_b r a - c''\log(a)+c''\log(r)}
         \qquad \text{for $A \in \Delta\mathcal{B}_{r}$},
   \end{align*}
   which we apply to \eqref{eq:S1r-S1rprime-decompose} to complete the proof.
\end{proof}

Combining \eqref{eq:contour-to-DeltaS1r}, \eqref{eq:estimate_S2r} and \eqref{eq:S1r-S1rprime} we conclude that there exist $\gamma,c > 0$, such that
\begin{equation*}
\begin{split}
\biggl|\int_{\mathbb{R}^2}&F(E_1,E_2)d\mu_\ell^r[b](E_1,E_2)  - \int_{\mathbb{R}^2}F(E_1,E_2)d\mu_\ell^{r'}[b](E_1,E_2)\biggr| \\
& \leq \sup_{z,z' \in \C_a} |F(z,z')| e^{-\gamma r a - c \log(a)+c \log(r)}.
\end{split}
\end{equation*}
In particular, it follows that $\int_{\mathbb{R}^2}F(E_1,E_2)d\mu_\ell^r[b](E_1,E_2)$ has a limit, which we denote by
\begin{equation*}
   \int_{\mathbb{R}^2}F(E_1,E_2)d\mu_\ell[b](E_1,E_2) :=
   \lim_{r \to \infty} \int_{\mathbb{R}^2}F(E_1,E_2)d\mu_\ell^r[b](E_1,E_2).
\end{equation*}
As the limit of a bounded sequence of (matrix-valued) Radon measures, it is
clear that $\mu_\ell[b]$ is again a Radon measure.

Finally, we establish the regularity of $\mu_\ell^r[b]$ and $ \mu_\ell[b]$ as
functions of $b \in \Gamma_{\tau(\ell)}$, where we recall that $\tau$ is the
transposition operator, $\tau(1) = 2$ and $\tau(2) = 1$. The statement that
\begin{equation*}
   b \mapsto \int_{\mathbb{R}^2} F(E_1,E_2) d\mu_{\ell}^r[b](E_1,E_2)
   \in C^n(\Gamma_{\tau(\ell)})
\end{equation*}
follows immediately from the resolvent representation \eqref{e:resolvent_form}
and the fact that $(z - H_\ell^r(b))^{-1}$ is $n$ times differentiable with
respect to $b$ (All operators involved here are finite-dimensional).

Thus, it remains only to show the regularity
\begin{equation} \label{eq:regularity-mulb}
   b \mapsto \int_{\mathbb{R}^2} F(E_1,E_2) d\mu_{\ell}[b](E_1,E_2)
   \in C^n_{\rm per}(\Gamma_{\tau(\ell)}).
\end{equation}

To that end, we consider the operator
$H_\ell(b) \in \mathcal{L}(\ell^2(\Omega)).$ Using Lemma \ref{lemma:thermo}, we have
\begin{equation*}
\begin{split}
\int_{\mathbb{R}^2} &F(E_1,E_2) d\mu_\ell[b](E_1,E_2)
 = \frac{-1}{4\pi }\oint_{z'\in\C_a}\oint_{z \in \C_a} \\
&\qquad F(z,z')\langle e_{0\alpha}|(z-H_\ell(b))^{-1} \partial_p H_\ell(b) (z'-H_\ell(b))^{-1}\partial_{p'} H_\ell(b)|e_{0\alpha} \rangle dzdz'.
\end{split}
\end{equation*}
We notice that differentiation of the resolvent $(z - H_\ell^r(b))^{-1}$ leads to products of the resolvent $(z-H_\ell^r(b))^{-1}$ and matrices of the form $\partial_{b_1}^{m_1}\partial_{b_2}^{m_2} H_\ell^r(b)$, all of which are well defined in the thermodynamic limit and have periodic limits with respect to $\Gamma_{\tau(\ell)}$. For an example, consider the derivative
\begin{equation*}
\begin{split}
\partial_{b_1}(z-H_\ell^r(b))^{-1} &= (z-H_\ell^r(b))^{-1} \partial_{b_1} H_\ell^r(b) (z-H_\ell^r(b))^{-1} \\
&\rightarrow (z-H_\ell(b))^{-1} \partial_{b_1} H_\ell(b) (z-H_\ell(b))^{-1}.
\end{split}
\end{equation*}
Hence $(z - H_\ell(b))^{-1}$ is a differentiable operator when acting on an element of the domain, and  we trivially find $\int F\mu_\ell[b] \in C^n_\per(\Gamma_{\tau(\ell)})$.

\subsection{Proof of Theorem \ref{thm:main}}
We recall that the current-current correlation measure for the finite system was defined through
\begin{equation}
\int_{\mathbb{R}^2}F(E_1,E_2)d\ccm^r(E_1,E_2) =\sum_{ii'}F(\eps_i,\eps_{i'}) \frac{1}{|\Omega_r|}\Tr[|v_{i}\rangle \langle v_i |\partial_p H^r  |v_{i'}\rangle \langle v_{i'} |\partial_{p'}H^r|].
\end{equation}
We can decompose this into local current-current correlation measures of the finite system by defining $\mu_{R\alpha}^r$ via
\begin{equation*}
\int_{\mathbb{R}^2} F(E_1,E_2)d \mu_{R\alpha}^r = \sum_{ii'}F(\eps_i,\eps_{i'})[|v_{i}\rangle \langle v_i |\partial_p H^r  |v_{i'}\rangle \langle v_{i'} |\partial_{p'}H^r|]_{R\alpha,R\alpha}.
\end{equation*}
Hence,
\begin{equation*}
\int_{\mathbb{R}^2} F(E_1,E_2) d\ccm^r(E_1,E_2) = \frac{1}{|\Omega_r|} \sum_{R\alpha \in \Omega_r} \int_{\mathbb{R}^2} F(E_1,E_2) d\mu_{R\alpha}^r(E_1,E_2).
\end{equation*}
We will also reserve the notation for $\Omega' \subset \Omega$ finite
\begin{equation*}
\int_{\mathbb{R}^2} F(E_1,E_2')d \mu_{\ell}^{\Omega'}[b] = \sum_{ii'}F(\eps_i,\eps_{i'}) \frac{1}{|\Omega'|}[|v_{i}\rangle \langle v_i |\partial_p H_\ell(b)|_{\Omega'}  |v_{i'}\rangle \langle v_{i'} |\partial_{p'}H_\ell(b)|_{\Omega'}|]_{R\alpha,R\alpha}.
\end{equation*}
Here, $(\varepsilon_i,v_i)$ are the eigenpairs for $H_\ell(b)|_{\Omega'}$.
We pick $D > 0$, and then consider $\bsigma^r$, where we wish to consider the limit $r \rightarrow \infty$. We have
\begin{equation*}
\begin{split}
\bsigma^r &=  \frac{1}{|\Omega_r|} \int_{\mathbb{R}^2} F_{\parelem}(E_1,E_2)d\ccm^r(E_1,E_2)\\
& =\int_{\mathbb{R}^2} \frac{1}{|\Omega_r|} F_{\parelem}(E_1,E_2) \biggl(\sum_{R\alpha \in \Omega_{r-D}} d\mu_{R\alpha}^r(E_1,E_2) + \sum_{R\alpha \in\Omega_r \setminus\Omega_{r-D}} d\mu_{R\alpha}^r(E_1,E_2)\biggr). \\
\end{split}
\end{equation*}
We define the domain $\Omega_R^r$ for $R \in \R_\ell$ such that
\begin{equation*}
\Omega_R^r = \biggl((\R_\ell\cap B_r - R) \times \A_\ell\biggr) \cup \biggl( \R_{\tau(\ell)}\cap B_r - R + \mod_{\tau(\ell)}(R))\times \A_{\tau(\ell)}\biggr).
\end{equation*}
For $|R| < r - D$,
\begin{equation*}
\begin{split}
&\biggl|\int_{\mathbb{R}^2} F_{\parelem}(E_1,E_2)\sum_{\alpha \in \A_\ell}d\mu_{R\alpha}^r(E_1,E_2)- \int_{\mathbb{R}^2} F_{\parelem}(E_1,E_2)d\mu_\ell^D[R](E_1,E_2)\biggr|\\
&=\biggl|\int_{\mathbb{R}^2} F_{\parelem}(E_1,E_2)d\mu_\ell^{\Omega_R^r}[R](E_1,E_2)- \int_{\mathbb{R}^2} F_{\parelem}(E_1,E_2)d\mu_\ell^D[R](E_1,E_2)\biggr| \\
& \lesssim e^{-\gamma \lambda D - c\log(\lambda)}.
\end{split}
\end{equation*}
The last line follows from (\ref{e:bound}), the fact that $\Omega_D \subset \Omega_R^r$,
and $F_{\parelem}(E_1,E_2)$ is analytic on $S_\lambda \times S_\lambda$. Using Theorem \ref{thm:ergodic}, we have
\begin{equation*}
\begin{split}
\limsup_{r \rightarrow \infty} &\biggl|\int_{\mathbb{R}^2} F_{\parelem}(E_1,E_2)\frac{1}{|\Omega_r|}\sum_{R\alpha \in \Omega_{r-D}} d\mu_{R\alpha}^r(E_1,E_2) \\
&\qquad - \int_{\mathbb{R}^2} F_{\parelem}(E_1,E_2)\nu\sum_{\alpha \in \A_\ell}\int_{\Gamma_{p(\ell)}}d\mu_\ell^D[b](E_1,E_2) \biggr|  \lesssim e^{-\gamma \lambda D - c\log(\lambda)}.
\end{split}
\end{equation*}
Further,
\begin{equation*}
\frac{1}{|\Omega_r|} \int_{\mathbb{R}^2} F_{\parelem}(E_1,E_2)\sum_{R\alpha \in\Omega_r \setminus\Omega_{r-D}} d\mu_{R\alpha}^r(E_1,E_2) \rightarrow 0
\end{equation*}
as $r \rightarrow \infty$ since $\frac{|\Omega_r \setminus \Omega_{r-D}|}{|\Omega_r|} \rightarrow 0$. Hence we have, letting $D \rightarrow \infty$,
\begin{equation*}
 \frac{1}{|\Omega_r|} \int_{\mathbb{R}^2} F_{\parelem}(E_1,E_2)d\ccm^r(E_1,E_2) \rightarrow \int_{\mathbb{R}^2} F_{\parelem}(E_1,E_2)d\mu(E_1,E_2) = \sigma.
\end{equation*}
This is the desired global thermodynamic result.  Finally,
\begin{equation*}
|\sigma-\sigma^r| \lesssim e^{-\gamma \lambda r - c \log(\lambda)+c\log(r)}
\end{equation*}
is a trivial application of Lemma \ref{thm:local}.


\section{Proof of Theorem \ref{thm:conductivity_coeffs_bound}}\label{sec:proofs_numerics}

\subsection{Approximation theory background}\label{ssec:approximation_theory}

This subsection briefly recalls some concepts from approximation theory and introduces the notation used in the remainder of this section. A textbook introduction to the topics discussed here can be found e.g.,\ in \cite{Tre13}.

\emph{Joukowsky map $\phi(z)$.}
The three-term recurrence relation \eqref{eqn:chebrecursion} for the Chebyshev polynomials $T_k(x)$ is equivalent to
\begin{equation}\label{eqn:chebyshev_recursion}
T_k\bigl(\phi(z)\bigr)
:=
\frac{z^{k} + z^{-k}}{2},
\qquad
\text{where}
\qquad
\phi(z) :=
\frac{z + z^{-1}}{2}
\end{equation}
is known as the Joukowsky map.
Since $\phi(z) = \phi\bigl(z^{-1}\bigr)$, {the inverse Joukowsky map $\phi^{-1}(x)$} has two branches related by $\phi^{-1}_\pm(x) = \big(\phi^{-1}_\mp(x)\big)^{-1}$.
Given any curve $b \subset \mathbb{C}$ connecting the two branch points $x = \pm 1$, we define
\[
\phi^{-1}_b(x) := x + \sqrt[b]{x^2-1}
,
\]
where $\sqrt[b]{x^2-1}$ denotes the branch of $\sqrt{x^2-1}$ with branch cut along $b$ and sign such that $\phi^{-1}_b(\infty) = \infty$.

\emph{Bernstein ellipses $E(\alpha)$ and parameter function $\alpha_b(x)$.}
The definition of the Bernstein ellipses $E(\alpha)$ in \eqref{eqn:Bernstein_ellipses} is equivalent to
\[
E(\alpha)
=
\{x \in \mathbb{C} \mid \alpha_{[-1,1]}(x) < \alpha\}
,
\]
where the parameter function $\alpha_b(x)$ is given by
\[
\alpha_b(x) := \log |\phi^{-1}_b(x)|
.
\]
This function satisfies the following properties.

\begin{lemma}\label{lem:Joukowsky_properties}
    $ $
    \begin{itemize}
        \item $\alpha_b(x) = 0$ for all $x \in [-1,1]$ and all branch cuts $b$.
        \item $\alpha_{[-1,1]}(x) \geq 0$ for all $x \in \mathbb{C}$.
        \item $\alpha_b(x+0n) = -\alpha_b(x-0n)$ for all $x \in b$ and all branch cuts $b$, where the notation $x \pm 0n$ indicates that we evaluate $\alpha_b(x)$ on different sides of the branch cut.
    \end{itemize}
\end{lemma}

\emph{Zero-width contours.}
In an abuse of notation, we define $\partial \gamma$ for curves $\gamma \subset \mathbb{C}$ as the counterclockwise contour around a domain of infinitesimal width.
For example,
\[
\partial [-1,1]
=
\bigl([-1,1]+0\iota\bigl) \cup \bigl([-1,1]-0\iota\bigl)
,
\]
where the signed zero in the imaginary part indicates which branch to evaluate for a function with branch cut along $[-1,1]$.
\begin{example}
    We have
    \begin{align*}
        \int_{\partial[-1,1]} \phi_{[-1,1]}^{-1}(x) \, dx
        &=
        \int_{\partial[-1,1]} \Bigl(x + \sqrt[{[-1,1]}]{x^2 - 1}\Bigr) \, dx
        \\&=
        \int_{1+0\iota}^{-1+0\iota} \Bigl( x + \iota \, \sqrt{1-x^2} \Bigr) \, dx
        +
        \int_{-1-0\iota}^{1-0\iota} \Bigl( x -\iota \, \sqrt{1-x^2} \Bigr) \, dx
        \\&=
        -2 \iota \int_{-1}^1 \sqrt{1-x^2} \, dx
        =
        -\pi \, \iota
        ,
    \end{align*}
    where $\sqrt{y}$ with $y > 0$ denotes the positive square root and the sign of $\sqrt[{[-1,1]}]{x^2 - 1} = \pm \iota \sqrt{1-x^2}$ (i.e.\ the $\pm$ in $\pm \iota \sqrt{1-x^2}$) has been determined as follows.
    \begin{itemize}
        \item $\sqrt[{[-1,1]}]{x^2 - 1}$ has no branch cut along $\partial [-1,1]$, and $\sqrt[{[-1,1]}]{x^2-1} \neq 0$ for $x \neq \pm 1$;
        hence the only $x \in \partial[-1,1]$ where $\sqrt[{[-1,1]}]{x^2 - 1}$ is allowed to change sign is $x = \pm 1$.
        The sign of $\sqrt[{[-1,1]}]{x^2 - 1}$ on $[-1,1]+0\iota$ is therefore equal to the sign of $\sqrt[{[-1,1]}]{(0+0\iota)^2 - 1}$, and the sign of $\sqrt[{[-1,1]}]{x^2 - 1}$ on $[-1,1]-0\iota$ is equal to the sign of $\sqrt[{[-1,1]}]{(0-0\iota)^2 - 1}$.

        \item The sign of $\sqrt[{[-1,1]}]{(0+0\iota)^2 - 1}$ must be equal to the sign of $\sqrt[{[-1,1]}]{x^2 - 1}$ in the limit $x \to +\infty\iota$ since $\sqrt[{[-1,1]}]{x^2 - 1}$ is nonzero, purely imaginary and does not have a branch cut along the ray $(0,\infty) \, \iota$.

        \item We must have $\sqrt[{[-1,1]}]{x^2 - 1} = \iota \sqrt{1-x^2}$ in the limit $x \to +\infty\iota$ since for the opposite sign we would obtain $\lim_{x \to +\infty\iota} \phi_{[-1,1]}^{-1}(x) = x + \sqrt[{[-1,1]}]{x^2 - 1} \to 0$, which contradicts the definition of $\phi^{-1}_{[-1,1]}(x)$.

        \item The sign of $\sqrt[{[-1,1]}]{(0-0\iota)^2 - 1}$ can be determined analogously.
    \end{itemize}
\end{example}

\emph{Exponential decay with asymptotic rate $\alpha$.}
Following the $\mathcal{O}_\varepsilon$ notation of \cite{Tre17}, we introduce $a_k \leq_\varepsilon C(\alpha) \, \exp(-\alpha k)$ as a shorthand notation for exponential decay with asymptotic rate $\alpha$, i.e.,
\[
a_k \leq_\varepsilon C(\alpha) \, \exp\bigl(-\alpha k\bigr)
\quad
:\iff
\quad
\forall \tilde \alpha < \alpha :
a_k \leq C(\tilde \alpha) \, \exp\bigl(-\tilde \alpha k\bigr)
.
\]
We further write $a_k \lesssim_\varepsilon \exp(-\alpha k)$ if the prefactor $C(\alpha)$ is irrelevant.

If $\lim_{\tilde \alpha \to \alpha} C(\tilde \alpha)$ exists and is bounded, then $a_k \leq_\varepsilon C(\alpha) \, \exp\bigl(-\alpha k \bigr)$ is equivalent to $a_k \leq C(\alpha) \, \exp(-\alpha k )$.
A typical example of a sequence $a_k \leq_\varepsilon C(\alpha) \, \exp(-\alpha k )$ is $a_k := k \, \exp(-\alpha k )$, in which case $C(\tilde \alpha) = \max_k k \, \exp\bigl(-(\alpha-\tilde\alpha) \, k\bigr)$ and $\lim_{\tilde \alpha \to \alpha} C(\tilde \alpha) = \infty$.
For the purposes of this paper, the distinction between ``$a_k \leq_\varepsilon C(\alpha) \exp(-\alpha k)$'' and ``$a_k \leq C \, \exp(-\alpha k)$ for some unspecified $C > 0$'' is required for correctness, but it is of little practical relevance.

\emph{Analyticity in two dimensions.}
The notion of analyticity can be extended to two-dimensional functions $f(z_1,z_2)$ as follows.

\begin{definition}\label{def:analyticity_2d}
    A function $f : \Omega \to \mathbb{C}$ with $\Omega \subset \mathbb{C}^2$ is called \emph{analytic} if $f(z_1,z_2)$ is analytic in the one-dimensional sense in each variable $z_1, z_2$ separately for every $(z_1,z_2) \in \Omega$.
\end{definition}

This definition deserves several remarks.

\begin{itemize}
    \item By a well-known result due to Hartogs (see e.g.\ \cite[Theorem 1.2.5]{Kra92}), a function $f(z_1,z_2)$ analytic in the above one-dimensional sense is continuous and differentiable in the two-dimensional sense.
    \item It is known that if $f(z_1,z_2)$ is analytic on an arbitrary set $\Omega \subset \mathbb{C}^2$, then there exists an {open} set $\Omega' \supset \Omega$ such that $f(z_1,z_2)$ is analytic on $\Omega'$.
    \item It is known that if $f(z_1, z_2)$ is analytic on the biannulus $A(r_1) \times A(r_2)$ with $A(r) := \{z \mid r^{-1} < |z| < r\}$, it can be expanded into a Laurent series
    \[
    f(z_1, z_2)
    =
    \sum_{k_1, k_2 = -\infty}^\infty a_{k_1 k_2} \, z_1^{k_1} \, z_2^{k_2}
    \]
    with coefficients given by
    \[
    a_{k_1 k_2}
    =
    - \frac{1}{4\pi^2}
    \int_{\gamma_2} \int_{\gamma_1} f(z_1, z_2) \, \, z_1^{-k_1-1} \, z_2^{-k_2-1} \, dz_1 \, dz_2
    \]
    for any bicontour $\gamma_1\times\gamma_2$ where $\gamma_\ell \subset A(r_\ell)$ are two rectifiable closed contours winding once around the origin, see e.g.\ \cite[Theorem 1.5.26]{Sch05}.
\end{itemize}

\subsection{Auxiliary results}

We next establish a contour-integral formula for the Chebyshev coefficients of analytic functions in Theorem \ref{thm:formula} and demonstrate in Theorem \ref{thm:bound} how this formula translates into a bound on the Chebyshev coefficients. Both results are straightforward generalizations of the one-dimensional results (see e.g., \cite{Tre13}), except that we allow for a general branch cut in Theorem \ref{thm:bound}, which will be important in Subsection \ref{ssec:aniso_chebdecay}.

\begin{theorem}\label{thm:formula}
    A function $f(x_1,x_2)$ analytic on $[-1,1]^2$ can be expanded into a Chebyshev series
    \begin{equation}\label{eqn:Chebyshev_series}
        f(x_1,x_2) =
        \sum_{k_1,k_2 = 0}^\infty c_{k_1 k_2} \, T_{k_1}(x_1) \, T_{k_2}(x_2)
        \quad \text{ on $[-1,1]^2$}
    \end{equation}
    with coefficients $c_{k_1k_2}$ given by
    \[
    c_{k_1k_2}
    =
    -\tfrac{(2 - \delta_{k_1 0}) (2 - \delta_{k_2 0})}{4\pi^2}
    \int_{\partial [-1,1]} \int_{\partial [-1,1]} \,
    f(x_1,x_2) \,
    \frac{T_{k_1}(x_1)}{\sqrt[{[-1,1]}]{x_1^2-1}} \,
    \frac{T_{k_2}(x_2)}{\sqrt[{[-1,1]}]{x_2^2-1}} \,
    dx_1 \, dx_2
    .
    \]
\end{theorem}

\begin{proof}
    $f(x_1,x_2)$ is analytic on $[-1,1]$ and $\phi(z)$ maps the unit circle $\{|z| = 1\}$ holomorphically onto $[-1,1]$, thus $f\bigl(\phi(z_1),\phi(z_2)\bigr)$ is analytic on $\{|z| = 1\}^2$ and can be expanded into a Laurent series
    \begin{equation}\label{Laurent series}
        f\bigl(\phi(z_1),\phi(z_2)\bigr)
        =
        \sum_{k_1,k_2 = -\infty}^\infty a_{k_1,k_2} \, z_1^{k_1} \, z_2^{k_2}
        \qquad
    \end{equation}
    with coefficients $a_{k_1k_2}$ given by
    \begin{equation}\label{eqn:Laurent coefficients}
        a_{k_1k_2}
        =
        -\frac{1}{4\pi^2} \int_{|z_2| = 1} \int_{|z_1| = 1} f\bigl(\phi(z_1),\phi(z_2)\bigr) \, z_1^{-k_1-1} \, z_2^{-k_2-1} \, dz_1 \, dz_2
        .
    \end{equation}
    Since $\phi(z) = \phi\bigl(z^{-1}\bigr)$, we conclude that $a_{k_1k_2}$ is symmetric about the origin in both $k_1$ and $k_2$, i.e.,\ $a_{k_1,k_2} = a_{-k_1,k_2}$ and $a_{k_1,k_2} = a_{k_1,-k_2}$. The terms in \eqref{Laurent series} can therefore be rearranged as a Chebyshev series in {$\phi(z_1)$, $\phi(z_2)$},
    \[
    \begin{aligned}
        f\big(\phi(z_1),\phi(z_2)\big)
        &=
        \sum_{k_1,k_2 = 0}^\infty
        (2-\delta_{k_1 0}) (2 - \delta_{k_2 0}) \,
        a_{k_1k_2} \,
        \frac{z_1^{k_1} + z_1^{-k_1}}{2} \,
        \frac{z_2^{k_2} + z_2^{-k_2}}{2}
        \\&=
        \sum_{k = 0}^\infty c_{k_1k_2} \, T_{k_1}\big(\phi({z_1})\big) \, T_{k_2}\big(\phi(z_2)\big)
        ,
    \end{aligned}
    \]
    which is \eqref{eqn:Chebyshev_series} with $c_{k_1k_2} := (2-\delta_{k_1 0}) (2 - \delta_{k_2 0}) \, a_{k_1k_2}$.
    The formula for the coefficients follows by substituting
    \[
    z_{\ell} \to \phi^{-1}_{[-1,1]}(x_{\ell})
    ,\qquad
    dz_{\ell} \to \frac{\phi_{[-1,1]}^{-1}(x_{\ell})}{\sqrt[{[-1,1]}]{x^2 - 1}} \, dx_{\ell}
    \quad\text{and}\quad
    \{|z_\ell| = 1\} \to \partial [-1,1]
    \]
    for $\ell = 1$ and $\ell = 2$ in the integrals in \eqref{eqn:Laurent coefficients} {and setting
    \[
    c_{k_1k_2}
    =
    (2 - \delta_{k_1 0}) \, (2 - \delta_{k_2 0}) \, \tfrac{1}{4} \, \bigl(a_{k_1,k_2} + a_{k_1,-k_2} + a_{-k_1,k_2} + a_{-k_1,-k_2}\bigr)
    .
    \]
    }
\end{proof}

\begin{theorem}\label{thm:bound}
    Let $\Omega_1, \Omega_2 \subseteq \mathbb{C}$ be two simply connected sets {with rectifiable boundaries $\partial \Omega_\ell$} such that both sets contain $-1$ and $1$.
    It then holds that
    \begin{align*}
        \hspace{6em}&\hspace{-6em}
        \left|
        \frac{(2 - \delta_{k_1 0}) (2 - \delta_{k_2 0})}{4\pi^2}
        \int_{\partial \Omega_2}
        \int_{\partial \Omega_1}
        f(x_1,x_2) \,
        \frac{T_{k_1}(x_1)}{\sqrt[b_1]{x_1^2-1}} \,
        \frac{T_{k_2}(x_2)}{\sqrt[b_2]{x_2^2-1}} \,
        dx_1 \, dx_2
        \right|
        \leq
        \ldots \\ &
        \leq
        C(\partial \Omega_1) \, C(\partial \Omega_2) \,
        \|f\|_{\partial \Omega_1 \times \partial \Omega_2} \,
        \exp\bigl(-\alpha_1 k_1 - \alpha_2 k_2\bigr)
    \end{align*}
    for all $k_1, k_2 \in \mathbb{N}$ and all branch cuts $\bigl(b_\ell \subset \Omega_{\ell}\bigr)_{\ell \in \{1, 2\}}$ connecting $-1,1$, where
    \[
    \Bigl(\alpha_{\ell} := \min \alpha_{b_{\ell}}(\partial \Omega_{\ell})\Bigr)_{\ell \in \{1, 2\}}
    \qquad
    \text{and}
    \qquad
    C(\partial\Omega)
    :=
    \frac{1}{\pi} \, \int_{\phi^{-1}_b(\partial\Omega)} \frac{|dz|}{|z|}
    .
    \]
\end{theorem}

\begin{proof}
    Reversing the substitutions in the proof of Theorem~\ref{thm:formula} transforms the expression on the left-hand side to \eqref{eqn:Laurent coefficients} up to a factor $(2-\delta_{k_1 0}) (2 - \delta_{k_2 0})$ and the integrals running over $\phi^{-1}_b(\partial\Omega_{\ell})$ instead of $\{|z_{\ell}| = 1\}$ for $\ell \in \{1,2\}$.
    The claim follows by bounding these integrals using H{\"o}lder's inequality.
\end{proof}

We illustrate the application of Theorems \ref{thm:formula} and \ref{thm:bound} by proving the following corollary which can be found e.g.,\ in \cite[Theorem 11]{BM48}, \cite[Lemma 5.1]{Tre17} and \cite[Theorem 11]{Boy09}.

\begin{corollary}\label{cor:ellipse_bound}
    The Chebyshev coefficients of a function $f(x_1,x_2)$ analytic on $E(\alpha_1) \times E(\alpha_2)$ are bounded by
    \begin{equation}\label{eqn:ellipse_bound}
        |c_{k_1k_2}|
        \lesssim
        4 \, \|f\|_{\partial E(\alpha_1) \times \partial E(\alpha_2)} \, \exp\bigl(-\alpha_1 k_1 - \alpha_2 k_2\bigr)
        \quad \text{ for all } k_1,k_2 \in \mathbb{N}
        .
    \end{equation}
\end{corollary}
\begin{proof}
    $f(x_1,x_2)$ is analytic on $[-1,1]^2 \subset E(\alpha_1) \times E(\alpha_2)$, thus Theorem \ref{thm:formula} states that we can expand $f(x_1,x_2)$ into a Chebsyhev series with coefficients given by
    \[
    c_{k_1k_2}
    =
    -\frac{(2 - \delta_{k_1 0}) (2 - \delta_{k_2 0})}{4\pi^2}
    \int_{{\partial [-1,1]}} \int_{{\partial [-1,1]}} \,
    f(x_1,x_2) \,
    \frac{T_{k_1}(x_1)}{\sqrt[{[-1,1]}]{x_1^2-1}} \,
    \frac{T_{k_2}(x_2)}{\sqrt[{[-1,1]}]{x_2^2-1}} \,
    dx_1 \, dx_2
    .
    \]
    The integrand in this expression is analytic on $x_1 \in E(\alpha_1) \setminus [-1,1]$ for any fixed $x_2 \in \partial [-1,1]$;
    hence by the one-dimensional Cauchy integral theorem we can move the contour in $x_1$ from $\partial [-1,1]$ to $\partial E(\tilde\alpha_1)$ for any $\tilde \alpha_1 < \alpha_1$, i.e.\ we have
    \[
    c_{k_1k_2}
    =
    \ldots
    \int_{\partial [-1,1]} \int_{\partial E(\tilde \alpha_1)} \,
    \ldots
    \, dx_1 \, dx_2
    .
    \]
    Arguing similarly in the second variable, we obtain
    \[
    c_{k_1k_2}
    =
    \ldots
    \int_{\partial E(\tilde \alpha_2)} \int_{\partial E(\tilde \alpha_1)} \,
    \ldots
    \, dx_1 \, dx_2
    \]
    for any pair $(\tilde \alpha_\ell < \alpha_\ell)_{\ell \in \{1,2\}}$,
    which by Theorem \ref{thm:bound} implies
    \[
    |c_{k_1,k_2}|
    \leq
    4 \,
    \|f\|_{\partial E(\tilde \alpha_1) \times \partial E(\tilde \alpha_2)} \,
    \exp\bigl(-\tilde \alpha_1 k_1 - \tilde \alpha_2 k_2\bigr)
    \]
    where we used
    $C\bigl(\partial E(\alpha)\bigr) = \frac{1}{\pi} \, \int_{|z| = \exp(\alpha)} \frac{|dz|}{|z|} = 2$
    and
    $\alpha_{[-1,1]}\bigl(\partial E(\alpha)\bigr) = \alpha$.
    This is precisely the bound \eqref{eqn:ellipse_bound}.
\end{proof}

\subsection{Chebyshev coefficients of the conductivity function}\label{ssec:aniso_chebdecay}

This subsection establishes the bound \eqref{eqn:conductivity_coeffs_bound} with explicit formulae for $\alpha_\mathrm{diag}(\zeta)$ and $ \alpha_\mathrm{anti}(\zeta)$. This will be done in two steps.
First, we will prove Theorem \ref{thm:frac_bound} below which bounds the Chebyshev coefficients of the factor $f(x_1,x_2) = \frac{1}{x_1 - x_2 + s}$ from \eqref{eqn:relaxation_factor} where we set $s := \omega + \iota\eta$ for notational convenience.
The extension to the conductivity function $F_\zeta$ will then be provided in Theorem \ref{thm:conductivity_coeffs_bound_2}.

We note that $\frac{1}{x_1 - x_2 + s}$ is analytic at all $x_1 \in \mathbb{C}$ except $x_1 = x_2 - s$, and likewise $\frac{1}{x_1 - x_2 + s}$ is analytic at all $x_2 \in \mathbb{C}$ except $x_2 = x_1 + s$.
The condition that $\frac{1}{x_1 - x_2 + s}$ be analytic on a domain $\Omega_1 \times \Omega_2$ is thus equivalent to $\bigl(\Omega_1 + s\bigr) \cap \Omega_2 = \{\}$, which is clearly the case for $\Omega_1 = \Omega_2 = [-1,1]$ and $\imag(s) \neq 0$, see Figure \ref{fig:contours_initial}.
By Theorem \ref{thm:formula}, we can thus expand $\frac{1}{x_1 - x_2 + s}$ into a Chebyshev series with coefficients given by
\begin{equation}\label{eqn:conductivity_coeffs}
c_{k_1k_2}
=
-\frac{(2 - \delta_{k_1 0}) (2 - \delta_{k_2 0})}{4\pi^2}
\int_{\partial \Omega_2} \int_{\partial \Omega_1} \,
\frac{1}{x_1 - x_2 + s} \,
\frac{T_{k_1}(x_1)}{\sqrt[b_1]{x_1^2-1}} \,
\frac{T_{k_2}(x_2)}{\sqrt[b_2]{x_2^2-1}} \,
dx_1 \, dx_2
\end{equation}
where for now $\Omega_1 = \Omega_2 = b_1 = b_2 = [-1,1]$.

Like in the proof of Corollary \ref{cor:ellipse_bound}, we will next use Cauchy's integral theorem repeatedly to move the contour domains $\Omega_1,\Omega_2$ to appropriate shapes and then employ Theorem \ref{thm:bound} to bound the Chebyshev coefficients.
To this end, let us introduce
\[
\hat\alpha_\mathrm{max}(s)
:=
\min\{\alpha_{[-1,1]}(\pm 1 - s)\}
=
\alpha_{[-1,1]}\bigl(1 - |\real(s)| - \iota\imag(s)\bigr)
,
\]
which is the parameter of the ellipse $E\bigl(\hat\alpha_\mathrm{max}(s)\bigr)$ penetrating the line $[-1,1] - s$ up to the endpoints $\pm 1 + s$  (see Figure \ref{fig:contours_definitions}), and let us denote by
\[
\hat D(s)
:=
\Bigl(E\big(\hat\alpha_\mathrm{max}(s)\big) + s\Bigr)
\cap
\bigl\{x \in \mathbb{C} \mid \imag(x) \leq 0\bigr\}
\]
the portion of $E\big(\hat\alpha_\mathrm{max}(s)\big) + s$ penetrating $[-1,1]$.
Since $\bigl([-1,1] + s\bigr) \cap \overline{\hat D(s) } = \{\}$ (see Figure \ref{fig:contours_definitions}),
we conclude that $\frac{1}{x_1 - x_2 + s}$ is analytic on $[-1,1] \times \big([-1,1] \cup \overline{\hat D(s)}\bigr)$, thus we can replace $\Omega_2 = [-1,1]$ with $\Omega_2 = [-1,1] \cup \hat D(s)$ without changing the value of the integral.
Similarly, we can move the branch cut $b_2 = [-1,1]$ to the lower boundary of $\Omega_2$,
\[
b_2
=
\hat b^\star(s)
:=
\bigl([-1,1] \setminus \hat D(s) \bigr)
\cup
\{x \in \partial \hat D(s) \mid \imag(x) < 0\}
,
\]
which in turn allows us to replace $\Omega_2 = [-1,1] \cup \hat D(s)$ with $\Omega_2 = \hat b^\star(s)$ and finally replace $\Omega_1 = [-1,1]$ with $\Omega_1 = E(\tilde \alpha_1)$ for any $\tilde \alpha_1 < \hat\alpha_\mathrm{max}(s)$, see Figure \ref{fig:contours_final}.
By Theorem \ref{thm:bound}, these final contours imply the bound
\begin{equation}\label{eqn:plane_bound}
|c_{k_1k_2}|
\lesssim_{\varepsilon}
\exp\bigl(-\hat\alpha_\mathrm{max}(s) \, k_1 - \hat\alpha_\mathrm{min}(s) \, k_2\bigr)
\end{equation}
where
\begin{equation}\label{eqn:hat_alpha_min}
\hat\alpha_\mathrm{min}(s)
:=
\min \alpha_{\hat b^\star(s)}\bigl(\partial \hat b^\star(s)\bigr)
=
-\max \alpha_{[-1,1]}\bigl(\hat b^\star(s)\bigr)
,
\end{equation}
(the second equality follows from Lemma \ref{lem:Joukowsky_properties}).
We note that the last expression in \eqref{eqn:hat_alpha_min} may be interpreted as minus the parameter of the smallest ellipse containing $\hat D(s)$, see Figure \ref{fig:contours_definitions}.

\begin{figure}
    \centering
    \subfloat[Initial contours]{\label{fig:contours_initial}{\includegraphics{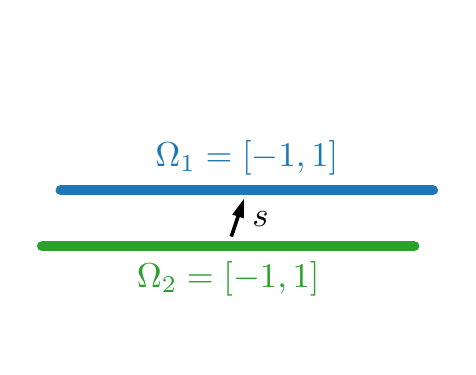}}}
    \hspace{0.2em}
    \subfloat[Final contours]{\label{fig:contours_final}{\includegraphics{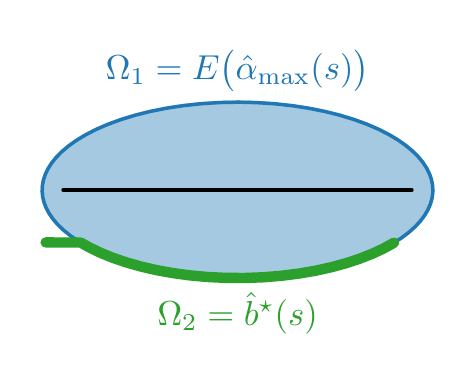}}}
    \hspace{0.2em}
    \subfloat[Definitions]{\label{fig:contours_definitions}{\includegraphics{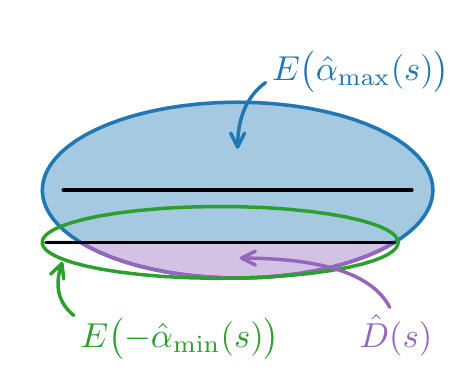}}}
    \caption{Illustration of the various definitions in Subsection \ref{ssec:aniso_chebdecay}.}
    \label{fig:ellipses}
\end{figure}

By the symmetry of $\frac{1}{x_1 - x_2 + s}$, the bound \eqref{eqn:plane_bound} also holds with the roles of $k_1,k_2$ interchanged, and since $\hat\alpha_\mathrm{max}(s) > 0$ but $\hat\alpha_\mathrm{min}(s) < 0$, we may summarize the two bounds with
\begin{equation}\label{eqn:conductivity_bound_cases}
|c_{k_1k_2}|
\lesssim_{\varepsilon}
\begin{cases}
    \exp\bigl(-\hat\alpha_\mathrm{max}(s) \, k_1 - \hat\alpha_\mathrm{min}(s) \, k_2\bigr)
    & \text{if } k_1 \geq k_2, \\
    \exp\bigl(-\hat\alpha_\mathrm{min}(s)\, k_1 - \hat\alpha_\mathrm{max}(s) \, k_2\bigr)
    & \text{if } k_1 \leq k_2. \\
\end{cases}
\end{equation}
Rewriting \eqref{eqn:conductivity_bound_cases} in the form \eqref{eqn:frac_bound}, we arrive at the following theorem.

\begin{theorem}\label{thm:frac_bound}
    The Chebyshev coefficients $c_{k_1k_2}$ of $f(x_1,x_2) := \frac{1}{x_1 - x_2 + s}$ with $\real(s) \in [-1,1]$ are bounded by
    \begin{equation}\label{eqn:frac_bound}
        |c_{k_1,k_2}|
        \lesssim_{\varepsilon}
        \exp\bigl(
            -\hat\alpha_\mathrm{diag}(s) \, (k_1+k_2)
            -\hat\alpha_\mathrm{anti}(s) \, |k_1 - k_2|
        \bigr)
    \end{equation}
    where
    \[
    \hat\alpha_\mathrm{diag}(s) := \tfrac{1}{2} \, \Big(\hat\alpha_\mathrm{max}(s) + \hat\alpha_\mathrm{min}(s)\Big)
    \quad\text{and}\quad
    \hat\alpha_\mathrm{anti}(s) := \tfrac{1}{2} \, \Big(\hat\alpha_\mathrm{max}(s) - \hat\alpha_\mathrm{min}(s)\Big)
    .
    \]
\end{theorem}

A closer inspection of the above argument reveals that the bound \eqref{eqn:frac_bound} holds for any function $f(x_1,x_2) = \frac{g(x_1,x_2)}{x_1 - x_2 + s}$ as long as $g(x_1,x_2)$ is analytic on $E\bigl(\hat\alpha_\mathrm{max}(s)\bigr)^2$,
and in particular it applies to the conductivity function $F_\zeta(E_1,E_2) = \frac{f_\mathrm{temp}(E_1,E_2)}{E_1 - E_2 + \omega + \iota\eta}$ if the singularities $S_\mathrm{temp}$ of $f_\mathrm{temp}(E_1,E_2)$ from \eqref{eqn:fermi_singularities} satisfy
\[
E\bigl(\hat\alpha_\mathrm{max}(\omega + \iota\eta)\bigr)^2 \cap S_\mathrm{temp} = \{\}
\quad\iff\quad
E\bigl(\hat\alpha_\mathrm{max}(\omega + \iota\eta)\bigr) \cap S^{(1)}_\mathrm{temp} = \{\}
,
\]
i.e.,\ if $\zeta$ is relaxation-constrained.
Furthermore, the argument and hence the bound \eqref{eqn:frac_bound} can be extended to the mixed- and temperature-constrained cases if we replace $\hat\alpha_\mathrm{max}(s)$ with
\begin{equation}\label{eqn:alpha_max}
\alpha_\mathrm{max}(\zeta)
:=
\min\bigl\{
    \alpha_{[-1,1]}(1 - |\omega| + \iota \eta)
    , \,
    \alpha_{[-1,1]}\bigl(E_F + \tfrac{\pi\iota}{\beta}\bigr)
\bigr\}
,
\end{equation}
which is the parameter of the blue ellipses in Figure \ref{fig:contours_with_temperature}.
This leads to new variables {$D(\zeta)$ and $b^\star(\zeta)$} defined analogously to {$\hat D(s)$ and $\hat b^\star(s)$}, respectively, but starting from $\alpha_\mathrm{max}(\zeta)$ instead of $\hat \alpha_\mathrm{max}(s)$, i.e.\
\begin{align*}
D(\zeta)
&:=
\Bigl(E\big(\alpha_\mathrm{max}(\zeta)\big) + \omega + \eta\iota\Bigr)
\cap
\bigl\{x \mid \imag(x) \leq 0\bigr\}
,\\
b^\star(\zeta)
&:=
\bigl([-1,1] \setminus D(\zeta) \bigr)
\cup
\{x \in \partial D(\zeta) \mid \imag(x) < 0\}
.
\end{align*}
Finally, we generalise $\hat \alpha_\mathrm{min}(s)$ to
\begin{equation}\label{eqn:alpha_min}
\alpha_\mathrm{min}(\zeta)
=
\min\bigl\{
\alpha_{b^\star(\zeta)}\bigl(x^\star(\zeta) + 0\iota\bigr)
,
\alpha_{[-1,1]}\bigl(E_F + \tfrac{\pi\iota}{\beta}\bigr)
\bigr\}
,
\end{equation}
where $x^\star(\zeta)$ is given by
\begin{equation}\label{eqn:xstar}
    x^\star(\zeta)
    :=
    \argmin_{x \in \partial E(\alpha_\mathrm{max}(\zeta)) + \omega + \iota \eta} \alpha_{b^\star(\zeta)}(x + 0\iota)
    .
\end{equation}
Note that $\alpha_\mathrm{min}(\zeta)$ is the parameter of the green ellipses in Figure \ref{fig:contours_with_temperature}, and $x^\star(\zeta)$ is indicated by the purple dots in Figure \ref{fig:contours_with_temperature}.

With the above notation, we can now formally describe the classification into relaxation-, mixed- and temperature-constrained parameters $\zeta$, and we can generalize Theorem \ref{thm:frac_bound} to Theorem \ref{thm:conductivity_coeffs_bound_2} below.

\begin{definition}\label{def:classification}
    \[
    \text{We call $\zeta$ }
    \begin{cases}
        \text{relaxation-constrained}
        & \text{if }
        \alpha_{[-1,1]}(1 - |\omega| + \iota \eta)
        \leq
        \alpha_{[-1,1]}\bigl(E_F + \tfrac{\pi\iota}{\beta}\bigr)
        ,\\
        \text{temperature-constrained}
        & \text{if }
        \alpha_{[-1,1]}\bigl(E_F + \tfrac{\pi\iota}{\beta}\bigr)
        \leq
        \alpha_{b^\star(\zeta)}\bigl(x^\star(\zeta) + 0\iota\bigr)
        ,\\
        \text{mixed-constrained}
        &
        \text{otherwise}
        .
    \end{cases}
    \]
\end{definition}

\begin{theorem}\label{thm:conductivity_coeffs_bound_2}
    The Chebyshev coefficients $c_{k_1k_2}$ of $F_\zeta(E_1,E_2)$ are bounded by
    \[
        |c_{k_1,k_2}|
        \lesssim_{\varepsilon}
        \exp\bigl(
            -\alpha_\mathrm{diag}(\zeta) \, (k_1+k_2)
            -\alpha_\mathrm{anti}(\zeta) \, |k_1 - k_2|
        \bigr)
    \]
    where
    \[
    \alpha_\mathrm{diag}(\zeta) := \tfrac{1}{2} \, \Big(\alpha_\mathrm{max}(\zeta) + \alpha_\mathrm{min}(\zeta)\Big)
    \quad\text{and}\quad
    \alpha_\mathrm{anti}(\zeta) := \tfrac{1}{2} \, \Big(\alpha_\mathrm{max}(\zeta) - \alpha_\mathrm{min}(\zeta)\Big)
    .
    \]
\end{theorem}

\subsection{Asymptotics}
\label{ssec:asymptotics}

To complete the proof of Theorem \ref{thm:conductivity_coeffs_bound}, it remains to {determine the asymptotic scaling of $\alpha_\mathrm{diag}(\zeta)$ and $\alpha_\mathrm{anti}(\zeta)$ and the asymptotic parameter classification.
We will do so in Subsubsections \ref{sssec:alpha_scaling} and \ref{sssec:classification} using the following auxiliary result.
}

\begin{lemma}\label{lem:Joukowsky_asymptotics}
    It holds that
    \begin{alignat}{2}
        \label{eqn:O_im}
        \alpha_{[-1,1]}(x)
        &=
        \Theta\big(|\imag(x)|\big)
        \qquad
        &&\text{for } x \to x^\star \text{ with } x^\star \in (-1,1)
        , \\
        \label{eqn:O_sqrt}
        \alpha_{[-1,1]}(x)
        &=
        \Theta\bigl(\sqrt{|x \mp 1|}\bigr)
        \qquad
        &&\text{for } x \to \pm 1 \text{ with } \pm\real(x)-1 \geq C |\imag(x)|
        .
    \end{alignat}
\end{lemma}
\begin{proof}
    \eqref{eqn:O_im}:
    $\alpha_{[-1,1]}(x) = \real\bigl(\log \phi^{-1}_{[-1,1]}(x) \bigr)$ is {symmetric about the real axis and} harmonic on either side of the branch cut {at any $x^\star \in (-1,1)$; hence we can write}
    \[
    a_{[-1,1]}(x)
    =
    a_{[-1,1]}(x^\star)
    +
    \tfrac{\partial a_{[-1,1]}}{\partial \real(x)}(x^\star) \, \real(x - x^\star)
    +
    \tfrac{\partial a_{[-1,1]}}{\partial \imag(x)}(x^\star + 0\iota) \, |\imag(x)|
    +
    \mathcal{O}\bigl(|x-x^\star|^2\bigr)
    .
    \]
    Since $\alpha_{[-1,1]}(x^\star) = 0$ for all $x^\star \in (-1,1)$, the constant term vanishes, and writing $\alpha_{[-1,1]}(x) = \big(\varphi^{-1}\bigl(\real(x),\imag(x)\bigr)\big)_1$ with
    \begin{align*}
        \varphi(\alpha,\theta)
        &:=
        \begin{pmatrix}
            \real\bigl(\phi(\exp(\alpha + \iota \, \theta))\bigr)
            \\
            \imag\bigl(\phi(\exp(\alpha + \iota \, \theta))\bigr)
        \end{pmatrix}
        =
        \begin{pmatrix}
            \cosh(\alpha) \cos(\theta)
            \\
            \sinh(\alpha) \sin(\theta)
        \end{pmatrix}
        ,
        \\
        \nabla \varphi(0,\theta)
        &=
        \begin{pmatrix}
            0 & {-\sin(\theta)} \\
            \sin(\theta) & 0
        \end{pmatrix}
        ,
    \end{align*}
    we conclude that
    \begin{align*}
        \tfrac{\partial \alpha_{[-1,1]}}{\partial \real(x)} (x^\star)
        &=
        \Big(\nabla\varphi(0,\theta^\star)^{-1}\Big)_{11}
        =
        0
        ,\\
        \tfrac{\partial \alpha_{[-1,1]}}{\partial \imag(x)} (x^\star+0\iota)
        &=
        \Big(\nabla\varphi(0,\theta^\star)^{-1}\Big)_{12}
        =
        \sin(\theta^\star)^{-1}
        \neq
        0
    \end{align*}
    where $\theta^\star = \acos\bigl(\real(x^\star)\bigr) \in {(0,\pi)}$.

    \eqref{eqn:O_sqrt}:
    We compute
    \begin{align*}
        \alpha\bigl(w^2 \pm 1\bigr)
        &=
        \Real\bigg(\log \Big( w^2 \pm 1 + \sqrt{(w^2 \pm 1)^2 - 1} \Big) \bigg)
        \\&=
        \Real\bigg(\log \Big(1 \mp w \, \sqrt{w^2 \pm 2} \mp w^2 \Big) \bigg)
        \\&=
        \real\Big(\mp \sqrt{\pm 2} \, w + \mathcal{O}\bigl(w^2\bigr)\Big)
        \quad \text{for } w \to 0
        ,
    \end{align*}
    where by $\sqrt{(w^2 \pm 1)^2 - 1}$ we mean a $w$-dependent combination of the two branches of the square-root function such that $\alpha\bigl(w^2 \pm 1\bigr)$ is harmonic around $w = 0$.
    The claim follows by substituting $w = \sqrt{x \mp 1}$ and noting that $\sqrt{\pm 2} \, \sqrt{x \mp 1}$ is bounded away from the imaginary axis as long as $x$ is bounded away from $(-1,1)$.
\end{proof}

\subsubsection{Scaling of $\alpha_\mathrm{diag}(\zeta)$ and $\alpha_\mathrm{anti}(\zeta)$}\label{sssec:alpha_scaling}

For temperature-constrained $\zeta$, we have
\[
\alpha_\mathrm{max}(\zeta)
=
\alpha_\mathrm{min}(\zeta)
=
\alpha_{[-1,1]}\bigl(E_F + \tfrac{\pi\iota}{\beta}\bigr)
=
\Theta(\beta^{-1})
\]
and hence
\[
\alpha_\mathrm{diag}(\zeta)
=
\alpha_\mathrm{max}(\zeta)
+
\alpha_\mathrm{min}(\zeta)
=
\Theta(\beta^{-1})
,\qquad
\alpha_\mathrm{anti}(\zeta)
=
\alpha_\mathrm{max}(\zeta)
-
\alpha_\mathrm{min}(\zeta)
=
0
.
\]

The remainder of this subsubsection establishes analogous estimates for relaxation- and mixed-constrained $\zeta$. In this case, we have
\begin{alignat}{2}
\alpha_\mathrm{max}(\zeta)
&=
\alpha_{[-1,1]}\bigl(x^\star(\zeta) - \omega - \eta\iota\bigr)
&&=
\Theta\bigl(\eta - \imag(x^\star(\zeta))\bigr)
\label{eqn:alpha_max_relax_mix}
, \\
\alpha_\mathrm{min}(\zeta)
&=
\alpha_{b^\star(\zeta)}\bigl(x^\star(\zeta) + 0\iota\bigr)
&&=
\Theta\bigl(\imag(x^\star(\zeta))\bigr)
,
\label{eqn:alpha_min_relax_mix}
\end{alignat}
which may be verified as follows.
\begin{itemize}
    \item The first expression for $\alpha_\mathrm{max}(\zeta)$ is an immediate consequence of the definition of $x^\star(\zeta)$ in  \eqref{eqn:xstar}.
    The second expression follows from \eqref{eqn:O_im}\footnote{
    We implicitly assume here that $x^\star(\zeta) - \omega - \eta\iota$ approaches some $x^\star \in (-1,1)$ in the limit considered in Theorem \ref{thm:conductivity_coeffs_bound} and not $x^\star(\zeta) - \omega - \eta\iota \to \pm 1$.
    The reader may easily convince themself that this is true using illustrations like the ones provided in Figure \ref{fig:contours_with_temperature}.
    A rigorous proof of this result is beyond the scope of this work.
    } after observing that $\imag\bigl(x^\star(\zeta)\bigr) < 0$ or $\imag\bigl(x^\star(\zeta)\bigr) < \eta$ and hence $\bigl|\imag(x^\star(\zeta)) - \eta\bigr| = \eta - \imag\bigl(x^\star(\zeta)\bigr)$.
    \item The first expression for $\alpha_\mathrm{min}(\zeta)$ is the definition of $\alpha_\mathrm{min}(\zeta)$ in \eqref{eqn:alpha_min} simplified for the relaxation- and mixed-constrained cases.
    The second expression follows by observing that $x^\star(\zeta) + 0\iota$ is always above the branch cut $b^\star(\zeta)$ and hence \eqref{eqn:O_im} applies without the absolute value on the right-hand side.
\end{itemize}

It follows from \eqref{eqn:alpha_max_relax_mix}, \eqref{eqn:alpha_min_relax_mix} that
\[
\alpha_\mathrm{diag}(\zeta)
=
\Theta\bigl(\eta - \imag(x^\star(\zeta))\bigr)
+
\Theta\bigl(\imag(x^\star(\zeta))\bigr)
=
\Theta(\eta)
,
\]
where we note that the two $\Theta\bigl(\imag(x^\star(\zeta))\bigr)$-terms indeed cancel since they arise from Taylor expansions of the same function $\alpha_{b}(x)$ around the same point $x = 0$.

To determine the asymptotic scaling of $\alpha_\mathrm{anti}(\zeta)$, we compare
\[
\alpha_\mathrm{max}(\zeta)
=
\left\{
\begin{aligned}
    \alpha_{[-1,1]}(1 - |\omega| + \iota\eta)
    &=
    \Theta(\eta^{1/2})
    &&
    \text{if $\zeta$ is relaxation-constrained}
    ,\\
    \alpha_{[-1,1]}(E_F + \tfrac{\pi\iota}{\beta})
    &=
    \Theta(\beta^{-1})
    &&
    \text{if $\zeta$ is mixed-constrained}
    .\\
\end{aligned}
\right.
\]
against \eqref{eqn:alpha_max_relax_mix} to conclude that
\[
\eta - \imag(x^\star(\zeta))
=
\begin{cases}
    \Theta(\eta^{1/2}) & \text{if $\zeta$ is relaxation-constrained}, \\
    \Theta(\beta^{-1}) & \text{if $\zeta$ is mixed-constrained}. \\
\end{cases}
\]
In the relaxation-constrained case, we thus have
\[
\Theta\bigl(\eta - \imag(x^\star(\zeta))\bigr)
=
-\Theta\bigl(\imag(x^\star(\zeta))\bigr) = \Theta(\eta^{1/2})
\]
and hence
\[
\alpha_\mathrm{anti}(\zeta)
=
\Theta\bigl(\eta - \imag(x^\star(\zeta))\bigr)
-
\Theta\bigl(\imag(x^\star(\zeta))\bigr)
=
\Theta\bigl(\eta^{1/2}\bigr)
.
\]
In the mixed-constrained case, we have
\[
\alpha_\mathrm{anti}(\zeta)
=
\Theta\bigl(\eta - \imag(x^\star(\zeta))\bigr)
-
\Theta\bigl(\imag(x^\star(\zeta))\bigr)
=
\mathcal{O}(\beta^{-1})
\]
where we used that $\eta \geq 0$ and hence $\eta - 2\imag(x^\star(\zeta)) \leq 2 \, \bigl(\eta - \imag(x^\star(\zeta))\bigr)$.
We remark that indeed $\alpha_\mathrm{anti}(\zeta) \neq \Theta(\beta^{-1})$ since $\alpha_\mathrm{anti}(\zeta) \to 0$ as $\beta$ approaches the finite value where $\zeta$ transitions into the temperature-constrained regime.

\subsubsection{Parameter classification}\label{sssec:classification}

In the limit considered in Theorem \ref{thm:conductivity_coeffs_bound}, Lemma \ref{lem:Joukowsky_asymptotics} yields
\[
\alpha_{[-1,1]}(1 - |\omega| + \iota \eta)
=
\Theta\bigl(\eta^{1/2}\bigr)
,\qquad
\alpha_{[-1,1]}\bigl(E_F + \tfrac{\pi\iota}{\beta}\bigr)
=
\Theta\bigl(\beta^{-1}\bigr)
;
\]
hence
\[
\Theta\bigl(\eta^{1/2}\bigr) \leq \Theta\bigl(\beta^{-1}\bigr)
\quad\iff\quad
\beta \lesssim \eta^{-1/2}
\]
if $\zeta$ is relaxation-constrained, and
\[
\Theta\bigl(\eta^{1/2}\bigr) \geq \Theta\bigl(\beta^{-1}\bigr)
\quad\iff\quad
\eta^{-1/2}
\lesssim
\beta
\]
if $\zeta$ is mixed-constrained.
To obtain the second bound for the mixed-constrained case, we observe that for given $E_F$, $\omega$ and $\eta$, the largest $\beta$ such that $\zeta$ is still mixed-constrained must be such that $E_F + \frac{\pi\iota}{\beta}$ and $\tfrac{\omega+\iota\eta}{2}$ lie on the same Bernstein ellipse; hence for mixed-constrained $\zeta$ we have
\[
\Theta(\beta^{-1})
=
\alpha_{[-1,1]}\bigl(E_F + \tfrac{\pi\iota}{\beta}\bigr)
\geq
\alpha_{[-1,1]}\bigl(\tfrac{\omega+\iota\eta}{2}\bigr)
=
\Theta(\eta)
\quad\iff\quad
\beta \lesssim \eta^{-1}
.
\]
Finally, for temperature-constrained $\zeta$, we must have
\[
\tfrac{\pi}{\beta} < \eta
\quad\iff\quad
\eta^{-1} \lesssim \beta
\]
(i.e.\ the first pole of the Fermi-Dirac function must lie between the two intervals in Figure \ref{fig:contours_with_temperature})
since otherwise $E(\eta) := E\bigl(\alpha_{[-1,1]}(\iota\eta)\bigr)$ and $E(0) = [-1,1]$ are two ellipses such that $E(\eta) + \omega + \iota\eta$ and $[-1,1]$ touch in a single point and neither $E(\eta)$ nor $[-1,1]$ intersect with the set of Fermi-Dirac singularities $S^{(1)}_\mathrm{temp}$, contradicting the assumption that $\zeta$ is temperature-constrained.

\section{Other Proofs: Numerics}

\subsection{Proof of Theorem \ref{thm:error_estimate}}\label{ssec:error_bound_proof}

Let us introduce
\[
b_{k_1k_2}
:=
\exp\bigl(
    -\alpha_\mathrm{max}(\zeta) \, k_1
    -\alpha_\mathrm{min}(\zeta) \, k_2
\bigr)
\]
with
\[
\alpha_\mathrm{max}(\zeta) := \alpha_\mathrm{diag}(\zeta) + \alpha_\mathrm{anti}(\zeta)
,\qquad
\alpha_\mathrm{min}(\zeta) := \alpha_\mathrm{diag}(\zeta) - \alpha_\mathrm{anti}(\zeta)
.
\]
Using Lemma \ref{lem:conductivity_approx_bound} and the bound \eqref{eqn:conductivity_coeffs_bound}, we obtain
\begin{align*}
\big| \tilde \sigma_\ell^r[b] - \sigma_\ell^r[b] \big|
&\lesssim
\sum_{(k_1,k_2) \in \mathbb{N}^2 \setminus K(\tau)} |c_{k_1k_2}|
\\&\leq
2\,C(\zeta) \sum_{(k_1,k_2) \in \mathbb{N}^2 \setminus K(r) \land k_1 \geq k_2} b_{k_1k_2}
\\&=
2\,C(\zeta) \,
\Bigg(
\underbrace{\sum_{k_2 = 0}^{K_2(\tau)-1} \sum_{k_1 = K_1(\tau,k_2)}^\infty b_{k_1k_2}}_{A}
+
\underbrace{\sum_{k_2 = K_2(\tau)}^{\infty} \sum_{k_1 = k_2}^\infty b_{k_1k_2}}_{B}
\Bigg)
\end{align*}
where
\[
K_2(\tau) := \left\lceil
    \frac{-\log(\tau)}%
    {2\,\alpha_\mathrm{diag}(\zeta)}
\right\rceil
,\qquad
K_1(\tau,k_2) := \left\lceil
    -
    \frac{\log(\tau) + \alpha_\mathrm{min}(\zeta) \, k_2}
    {\alpha_\mathrm{max}(\zeta)}
\right\rceil
.
\]
We then compute
\begin{align*}
    A
    &=
    \sum_{k_2 = 0}^{K_2(\tau)-1} \exp\bigl(-\alpha_\mathrm{min}(\zeta) \, k_2\bigr)
    \sum_{k_1 = K_1(\tau,k_2)}^\infty \exp\bigl(-\alpha_\mathrm{max}(\zeta) \, k_1\bigr)
    \\&\leq
    \sum_{k_2 = 0}^{K_2(\tau)-1} \exp\bigl(-\alpha_\mathrm{min}(\zeta) \, k_2\bigr)
    \,
    \frac{
        \tau \,
        \exp\bigl(\alpha_\mathrm{min}(\zeta) \, k_2\bigr)
    }{
        1 - \exp\bigl(-\alpha_\mathrm{max}(\zeta)\bigr)
    }
    \\&=
    \frac{K_2(\tau)}{1 - \exp\bigl(-\alpha_\mathrm{max}(\zeta)\bigr)} \, \tau
    \\&=
    \mathcal{O}\bigg(\alpha_\mathrm{diag}(\zeta)^{-1} \, \alpha_\mathrm{anti}(\zeta)^{-1} \, \tau \, \log(\tau)\bigg)
    \intertext{and}
    B
    &=
    \sum_{k_2 = K_2(\tau)}^{\infty} \exp\bigl(-\alpha_\mathrm{min}(\zeta) \, k_2 \bigr)
    \sum_{k_1 = k_2}^\infty \exp\bigl(-\alpha_\mathrm{max}(\zeta) \, k_1 \bigr)
    \\&=
    \sum_{k_2 = K_2(\tau)}^{\infty} \exp\bigl(-\alpha_\mathrm{diag}(\zeta) \, k_2 \bigr) \,
    \frac{1}{1 - \exp\bigl(-\alpha_\mathrm{max}(\zeta)\bigr)}
    \\&\leq
    \frac{\tau}{1 - \exp\bigl(-\alpha_\mathrm{diag}(\zeta)\bigr)}
    \,
    \frac{1}{1 - \exp\bigl(-\alpha_\mathrm{max}(\zeta)\bigr)}
    \\&=
    \mathcal{O}\bigg(\alpha_\mathrm{diag}(\zeta)^{-1} \, \alpha_\mathrm{anti}(\zeta)^{-1} \, \tau \bigg)
    ,
\end{align*}
where in the last steps for both terms we used that $\alpha_\mathrm{diag}(\zeta) = \mathcal{O}\bigl(\alpha_\mathrm{anti}(\zeta)\bigr)$ (cf.\ Theorem \ref{thm:conductivity_coeffs_bound}) and hence $\alpha_\mathrm{max}(\zeta) = \Theta\bigl(\alpha_\mathrm{anti}(\zeta)\bigr)$.

\subsection{Inverse of $\varepsilon = \tau \, |\log(\tau)|$}\label{ssec:tlogt_inverse}

This subsection establishes the following result.

\begin{theorem}
    Let $\varepsilon,\tau \in (0,\infty)$ be such that $\varepsilon = \tau \, |\log\tau|$.
    It then holds
    \[
    \tau = \frac{\varepsilon}{|\log\varepsilon|} \, \bigl(1 + o(1)\bigr)
    \qquad
    \text{for}
    \qquad
    \varepsilon \to 0
    .
    \]
\end{theorem}

\begin{proof}
    Dividing $\varepsilon = \tau \, |\log\tau|$ by $|\log\varepsilon| = \bigl|\log\tau + \log\log\tau|$, we obtain
    \[
    \frac{\varepsilon}{|\log\varepsilon|}
    =
    \tau \, \frac{1}{\bigl|1 + \frac{\log\log\tau}{|\log\tau|}\bigr|}
    \quad
    \iff
    \qquad
    \tau
    =
    \frac{\varepsilon}{|\log\varepsilon|} \, \bigl|1 + \tfrac{\log\log\tau}{|\log\tau|}\bigr|
    .
    \]
    The claim follows after noting that $\tau \, |\log(\tau)|$ is monotonically increasing in $\tau$ and hence $\tau \to 0$ for $\varepsilon \to 0$.
\end{proof}

\subsection{Proof of Theorem~\ref{thm:pole_expansion}}\label{ssec:pole_expansion_proof}

According to Riemann's removable singularity theorem in higher dimensions (see e.g.\ \cite[Thm.\ 4.2.1]{Sch05}), the function
\begin{equation}\label{eqn:remainder_formula}
    R(E_1,E_2)
    =
    \bigl(E_1 - E_2 + \omega + \iota\eta\bigr) \, F_\zeta(E_1,E_2)
    -
    \frac{1}{\beta} \, \frac{1}{(E_1 - z) \, (E_2 - z)}
\end{equation}
with $z := \frac{\pi\iota}{\beta}$ can be analytically continued to
\[
\mathcal{S}_z :=
\Bigl(\{z\} \times \bigl(\mathbb{C} \setminus \mathcal{S}_{\beta,E_F}\bigr)\Bigr) \cup \Bigl(\bigl(\mathbb{C} \setminus \mathcal{S}_{\beta,E_F}\bigr) \times \{z\}\Bigr)
\]
if $R(E_1,E_2)$ is bounded on this set, or equivalently if
\begin{equation}\label{eqn:residue_removed}
    \lim_{E_1 \to z} (E_1 - z) \, R(E_1,E_2)
    =
    0
\end{equation}
for some arbitrary $E_2 \in \mathbb{C}\setminus\mathcal{S}_{\beta,E_F}$ and likewise with the roles of $E_1$ and $E_2$ interchanged.
In order to verify \eqref{eqn:residue_removed}, we compute
\begin{align}
    \lim_{E_1 \to z} (E_1-z) \, f_\mathrm{temp}(E_1,E_2)
    &=
    \lim_{E_1 \to z} (E_1-z) \, \frac{f_{\beta,E_F}(E_1) - f_{\beta,E_F}(E_2)}{E_1 - E_2}
    \\ &=
    \frac{1}{z - E_2} \,
    \lim_{E_1 \to z} \frac{E_1-z}{1 + \exp\bigl(\beta \, (E_1 - E_F)\bigr)}
    \\&=
    \frac{1}{\beta}
    \,
    \frac{1}{E_2 - z}
    \label{eqn:temperature_residue}
\end{align}
where on the last line we used L'H\^opital's rule to determine the limit.
It follows from \eqref{eqn:temperature_residue} that for $E_1 \to z$, the first and second term in \eqref{eqn:remainder_formula} cancel and hence \eqref{eqn:residue_removed} holds.
The transposed version of \eqref{eqn:residue_removed} follows from the symmetry of \eqref{eqn:remainder_formula}, thus we conclude that $R(E_1,E_2)$ can indeed be analytically continued to $\mathcal{S}_z$.
Theorem~\ref{thm:pole_expansion} then follows by rewriting \eqref{eqn:pole_expansion} in the form \eqref{eqn:remainder_formula} and applying the above argument to each of the terms in the sum over $Z_k$.

\bibliographystyle{abbrv}
\bibliography{dos}

\end{document}